\documentclass{article}

\usepackage[utf8]{inputenc}
\usepackage{amsfonts}
\usepackage{amsmath}
\usepackage{amssymb}
\usepackage{mathrsfs}
\usepackage{xcolor}

\usepackage{amsthm}
\usepackage{bbm}

\newcommand{\diver}{\mathop{\rm div}\nolimits}

\theoremstyle{plain}
\newtheorem{lemma}{Lemma}
\newtheorem{theorem}{Theorem}

\newtheorem{theolettre}{Theorem}

\newtheorem{proposition}{Proposition}
\newtheorem{definition}{Definition}
\newtheorem{corollary}{Corollary}

\theoremstyle{remark}

\newtheorem{formula}{Formula}

\usepackage[style=numeric,backend=bibtex,maxbibnames=9]{biblatex}
\bibliography{articlePMB}

\usepackage{tikz}
\usepackage{tikz-3dplot}
\usepackage{subfig}
\usepackage{float}
\usetikzlibrary{backgrounds}

\tikzset{
	master/.style={
		baseline={-120 pt},
		execute at end picture={
			\coordinate (lower right) at (current bounding box.south east);
			\coordinate (upper left) at (current bounding box.north west);
		}
	},
	slave/.style={
		execute at end picture={
			\pgfresetboundingbox
			\path (upper left) rectangle (lower right);
		}
	}
}

\title{Mixing for the primitive equations under bounded non-degenerate noise}
\author{Pierre-Marie Boulvard}
\begin{document}

\maketitle
\begin{abstract}
	We study the stochastic 3D primitive equations of the atmospheric mechanics. We consider them under a bounded and non-degenerate noise, which is statistically periodic in time with period $1$. In such a case we prove that the associated integer-time Markov chain is exponentially mixing, which means that there exists a unique stationary measure to which the laws of all trajectories of this Markov chain converge exponentially fast.
	
	%After stating some classical results for the deterministic 3D primitive equations with bounded in time noise, we expand the definition on the system to allow forces that are square-integrable in time, albeit close to bounded in time forces. In this new setting we can use Theorem 1.3 in \cite{ART7} to prove that our system is mixing.
\end{abstract}
\tableofcontents
\newpage

\section{Introduction}
%%%Remplace la description physique
In the space $\mathbb{R}^3$, %a cuboid $\mathbb{O}:=[0,L]\times[0,L]\times[-h,0]$ 
 provided with a system of coordinates $\zeta:=(x,y,z)$, we consider the system of primitive equations of atmospheric mechanics (see Section 2.3 in \cite{BOK10}). This set of equations is of crucial importance for meteorology as it outlays the evolution of the cinematic fields of the atmosphere. Normally it describes the two-dimensional horizontal velocity field coupled with a temperature one. Here though, the latter will be omitted for simplicity. % It is a simplification of a 3D Navier-Stokes equations using the fact that due to the shallowness of the atmosphere, we may consider that there is no vertical hydrodynamic pressure, for lack of enough length for it to acquire a reasonnable importance: this is the hydrostatic approximation
 %Here we shall study a simplified version of it. 
 Our results remain true for the complete system, with a very similar but more cumbersome proof.
 
Thus cinematic field $u=(u_1,u_2,u_3)$ abides by the following system:
%%% A remettre si description physique
%Henceforth $u$ evolution abides by the below written primitive equations:

\begin{align}
    \label{eqinit}
    \frac{\partial u_k}{\partial t}-\Delta u_k+\sum_{j=1}^3 \partial_j(u_ju_k)+\partial_k p & =f_k, & k=1,2,
    \\
    \label{condiv}
    \diver u=\sum_{j=1}^3 \partial_j u & =0,
\end{align}
where the pressure $p$ does not depend on $z$, $p=p(x,y,t)$.%, $(x,y)\in \mathbb{T}_L^2:=[0,L]\times[0,L]$.
%therefore this pressure depends only on the position on $\mathbb{T}_L^2:=[0,L]\times[0,L]$.
 We supplement the equations with:
\begin{align*}
%\label{cb}
	& \text{for } j=1 \text{ or } 2,
	\\
	& u_j(x+L,y,z,t)=u_j(x,y,z,t),\quad u_j(x,y+L,z,t)=u_j(x,y,z,t), %&&\text{(i.e periodic in $x$ and $y$)}
	\\
	& u_j(x,y,z+h,t)=u_j(x,y,z,t),\quad u_j(x,y,-z,t)=u_j(x,y,z,t). %&&\text{(i.e periodic, even in $z$)}
\end{align*}
Thus $u_1$ and $u_2$ are periodic in $x$, $y$ and $z$ and even in $z$. Moreover,
%	& u_2(x+L,y,z,t)=u_2(x,y,z,t),\quad u_2(x,y+L,z,t)=u_2(x,y,z,t) &&\text{(i.e periodic in $x$ and $y$)}
%	\\
%	& u_2(x,y,z+h,t)=u_2(x,y,z,t),\quad u_2(x,y,-z,t)=u_2(x,y,z,t) &&\text{(i.e periodic, even in $z$)}
%	\\
\begin{align*}
	& u_3(x+L,y,z,t)=u_3(x,y,z,t),\quad u_3(x,y+L,z,t)=u_3(x,y,z,t), %&&\text{(i.e periodic in $x$ and $y$)}
	\\
	& u_3(x,y,z+h,t)=u_3(x,y,z,t),\quad u_3(x,y,-z,t)=-u_3(x,y,z,t), %&&\text{(i.e periodic, odd in $z$)}
\end{align*}
i.e $u_3$ is periodic in $x$, $y$ and $z$ and odd in $z$.
%%% Ancienne version
%\begin{align*}
%%\label{cb}
%& \text{for } j=1 \text{ or } 2
%\\
%& u_j(x+L,y,z,t)=u_j(x,y,z,t),\quad u_j(x,y+L,z,t)=u_j(x,y,z,t) &&\text{(i.e periodic in $x$ and $y$)}
%\\
%& u_j(x,y,z+h,t)=u_j(x,y,z,t),\quad u_j(x,y,-z,t)=u_j(x,y,z,t) &&\text{(i.e periodic, even in $z$)}
%\\
%%	& u_2(x+L,y,z,t)=u_2(x,y,z,t),\quad u_2(x,y+L,z,t)=u_2(x,y,z,t) &&\text{(i.e periodic in $x$ and $y$)}
%%	\\
%%	& u_2(x,y,z+h,t)=u_2(x,y,z,t),\quad u_2(x,y,-z,t)=u_2(x,y,z,t) &&\text{(i.e periodic, even in $z$)}
%%	\\
%&\text{ and}
%\\
%& u_3(x+L,y,z,t)=u_3(x,y,z,t),\quad u_3(x,y+L,z,t)=u_3(x,y,z,t) &&\text{(i.e periodic in $x$ and $y$)}
%\\
%& u_3(x,y,z+h,t)=u_3(x,y,z,t),\quad u_3(x,y,-z,t)=-u_3(x,y,z,t) &&\text{(i.e periodic, odd in $z$)}.
%\end{align*}

The goal of this work is to study some stochastic aspects of our system:
%, in the case of a random, bounded, right-hand term, which is statistically periodic in time.
given the observation-induced uncertainties in the empirical parameters of the problem described by the primitive equations, 
%for example in the initial condition, 
it is natural to consider its probabilistic side and study the primitive equations with a stochastic force $f$, and maybe a random initial data $v_0$. %An important point in this setting is to see whether or not the system defined by $\eqref{init}$ is mixing.
Due to the complex nature of the primitive equation however, not every random force $f$ shall fit our purposes: in particular, the classical white-in-time force is not suitable for our considerations, for lack of good energy estimates, something which we will come back to later, when laying out the existing results on mixing equations.

Therefore, we study here the primitive equations stirred by bounded random forces, statistically periodic in time given by non-degenerate random Haar series which we call red forces or red noise. This will be defined later on in section $\ref{rednoisesec}$.

\subsection{The reduced equation}
We first simplify $\eqref{eqinit}$ noting that $u_3$ can be deduced from the other two components of the velocity field using $\eqref{condiv}$.
Indeed, since $u_3$ is odd, periodic, then
\begin{equation*}
u_3(x,y,kh)=0,\quad k\in\mathbb{Z}.
\end{equation*}
Thus, we may write that
\begin{equation*}
u_3(\zeta, t)
%= u_3(x,y,kh/2,t)+\int_{kh/2}^z \partial_z u_3(x,y,\alpha,t)d\alpha
=-\int_{-h}^z \diver_2(u_1,u_2)(x,y,\alpha,t)d\alpha,
\end{equation*}
where $\diver_2(u):=\partial_x u_1+\partial_y u_2$. 
%\begin{equation*}
%    \frac{\partial u_k}{\partial t}-\Delta u_k+\sum_{j=1}^2 u_j\partial_j u_k-(\int_{-h}^z \diver_2(u))\partial_z u_k+\partial_k p =f_k,\quad k=1,2
%\end{equation*}
So, denoting  $v:=(u_1,u_2)$ we get the equation for $v$

% \begin{equation}
%     \label{eqtrans}
%     \frac{\partial v_k}{\partial t}-\Delta v_k+\sum_{j=1}^2 v_j\partial_j v_k-(\int_{-h}^z \diver_2(v))\partial_z v_k+\partial_k p =f_k\text{   } k=1,2
% \end{equation}

% équation que l'on réécrit
\begin{equation}
    \label{init}
    \frac{d v}{d t}-\Delta v+(v\cdot\nabla_2)v-\left(\int_{-h}^z \diver_2 v(\cdot,\cdot,\xi,\cdot)d\xi\right)\frac{\partial v}{\partial z}+\nabla_2 p= f, 
\end{equation}
%
%In a weak sense that we shall introduce later on this equation may also be written the following way
%
% \begin{equation}
%     \label{eqfb}
%     \frac{d v}{d t}-\Delta v+(v\cdot\nabla_2)v-\Big(\int_{-h}^z \diver_2 v\Big)\frac{\partial v}{\partial z}= f 
% \end{equation}
%
% Let us then focus on the boundary conditions. The boundary may be described as follows:
% \begin{itemize}
%     \item $\Gamma_s$ stands for the side of our domain $\Gamma_s:=\partial\mathbb{T}_L^2\times[-h;0]$,
%     \item $\Gamma_t$ is its ceiling 
%     $\Gamma_t:=\mathbb{T}_L^2\times\{0\}$,
%     \item $\Gamma_b$ is its floor
%     $\Gamma_b:=\mathbb{T}_L^2\times\{-h\}$.
% \end{itemize}
with the following boundary conditions:
%We then have the following boundary conditions:
% \begin{align}
%     \label{cb}
%     \frac{\partial v}{\partial }=0 && u_3=0 && \text{sur } \Gamma_t
%     \\
%     v=0 && u_3=0 && \text{sur }\Gamma_b
%     \\
%     v=0 && &&\text{sur }\Gamma_s 
% \end{align}
% \begin{itemize}
%     \label{cb}
%     \item Periodic in $x$ and $y$ i.e
%     \begin{equation*}
%         v(x+L,y,z,t)=v(x,y,z,t),\quad v(x,y+L,z,t)=v(x,y,z,t)
%     \end{equation*}
%     \item Dirichlet zero in z i.e
%     \begin{equation*}
%         v(x,y,-h,t)=v(x,y,0,t)=0
%     \end{equation*}
% \end{itemize}
%\begin{itemize}
%    \label{cb}
%    \item Periodic in $x$ and $y$ i.e
%    \begin{equation*}
%        v(x+L,y,z,t)=v(x,y,z,t),\quad v(x,y+L,z,t)=v(x,y,z,t)
%    \end{equation*}
%    \item Periodic, even in z i.e
%    \begin{equation*}
%        v(x,y,z+h,t)=v(x,y,z,t),\quad v(x,y,-z,t)=v(x,y,z,t)
%    \end{equation*}
%    \item Therefore, for us to wield the Poincaré's inequality we must add a zero-mean value condition
%    \begin{equation*}
%        \int_{\mathbb{O}} v=0
%    \end{equation*}
%    \item In addition, $v$ follows an incompressibility condition derived from $div u=0$. Indeed by integrating the latter in $z$ and adding the boundary conditions, one gets
%    \begin{equation*}
%        \diver_2\int_{-h}^0 v=0
%    \end{equation*}
%\end{itemize}
\begin{align}
	\label{cb1}
	& v(x+L,y,z,t)=v(x,y+L,z,t)=v(x,y,z+h,t)=v(x,y,z,t), %&&\text{(i.e periodic in $x$ and $y$)}
	%\tag{CB1}
	\\
	\label{cb2}
	 & v(x,y,-z,t)=v(x,y,z,t). %&&\text{(i.e periodic, even in $z$)}.
	%\tag{CB2}
\end{align}
%\begin{remark}
%	We may equivalently study the dynamics of the velocity field on a three-dimensional torus, keeping however the parity .
%\end{remark}
Note that since $v(x,y,z,t)$ satisfies $\eqref{cb1}$, $\eqref{cb2}$, we may consider it as an even in $z$, non-autonomous vector field on the torus, 
\begin{equation*}
	v:\mathbb{O}\times\mathbb{R}^+\rightarrow \mathbb{R}^2,\quad \mathbb{O}=(\mathbb{R}/L\mathbb{Z})^2\times(\mathbb{R}/h\mathbb{Z}).
\end{equation*}
 Denote $\mathbb{T}_L^2=(\mathbb{R}/L\mathbb{Z})^2$ and $\mathbb{T}_h=\mathbb{R}/h\mathbb{Z}$ then $\mathbb{O}=\mathbb{T}_L^2\times \mathbb{T}_h$.
%therefore the domain $\mathbb{O}$ will be considered in the following way
%$\mathbb{O}=  \mathbb{T}^3$; besides we denote $\zeta=(x,y,z)$.
%For us to be able to wield Poincaré's inequality, we add a zero mean value condition
%\begin{equation*}
%	\label{0M}
%	\int_{\mathbb{O}} v=0 \tag{0M}
%\end{equation*}
In addition, $v$ meets an incompressibility condition derived from $\eqref{condiv}$. Indeed by integrating the latter in $z$ and adding the boundary conditions, setting 
\begin{equation*}
	\bar{v}:=\int_{-h}^0 v(\cdot,\cdot,\xi,\cdot)d\xi,
\end{equation*}
one gets
\begin{equation}
	\label{0D}
	\diver_2\bar{v}=0. %\tag{0D}
\end{equation}
We now add some initial conditions
\begin{equation}
    \label{ci}
    v(\cdot,0)=v_0,
\end{equation}
with $v_0$ displaying, in addition to $\eqref{cb1}$, $\eqref{cb2}$ and $\eqref{0D}$, the zero-meanvalue property:
\begin{equation}
\label{ci0M}
\int_{\mathbb{O}} v_0d\zeta=
%\int_{\mathbb{O}} v_0(x,y,z)dxdydz=
0.
\end{equation}
As we furthermore require that $f$ follows 
\begin{equation}
\label{cf0M}
\left(\int_{\mathbb{O}}fd\zeta\right)(t) =
%\int_{\mathbb{O}}f(x,y,z,t)dxdydz=
0,\quad \forall t\geq 0,
\end{equation}
we get that the solution $v$ of $\eqref{init}$, also satisfies %under conditions $\eqref{ci0M}$ and $\eqref{cf0M}$ satisfies
\begin{equation}
\label{0M}
\left(\int_{\mathbb{O}} v d\zeta\right)(t)
%=\int_{\mathbb{O}} v(x,y,z,t)dxdydz
=0,\quad\forall t\geq 0. %\tag{0M}
\end{equation}
%$\diver_2\int_{-h}^0 v_0(\cdot,\cdot,\xi)d\xi=0$.

%\subsection{History of the deterministic problem}
Let us now briefly discuss some major milestones in the study of this problem: the groundwork was laid in \cite{ART10} in a much more general and complicated physical setting, the domain being a manifold and the density of the fluid not being constant. In addition to defining the natural spaces for the solutions the authors introduce a variational setting of the problem and solve the problem, albeit in a weak sense.
%formally that is with solutions defined as distributions in space.
To reach a better regularity for solutions the main obstacle to overcome is the non-linear term and in particular its $z$ component. 
A first way to circumvent this difficulty, introduced in \cite{ART20}, is to write a solution of the primitive equations as a deviation from a solution of the associated Stokes system. Then the boundedness of our solution in the space $H^1$ can be derived using the regularity of the solution of the Stokes system. This approach allows one to prove, among other things, a local existence of a solution of the primitive equations where the time of existence depends on the norm of the solution of the Stokes system.
Supplementing this idea with a small depth hypothesis the authors of \cite{ART11} are able to prove the existence of a global solution.
A major breakthrough was then achieved in \cite{ART1} where a careful study of the $L^6$ a priori estimates 
%for a purported solution 
implies the  boundedness of the derivatives of a solution and thus proves the well-posedness of the problem. Further regularity results for the solutions were obtained first for the second order Sobolev norm in \cite{ART18} and then for higher Sobolev norms in \cite{ART9}.
The notation and analytic considerations of our article follow that of \cite{ART1} and \cite{ART9}; the setting of the latter is close to ours.

Now 
%before looking into the random part of the problem,
 let us set up the exact setting in which the equations will be considered and state the well-posedness of the corresponding deterministic problem.

%\subsection{Notations}
First let us re-write $\eqref{init}$ more compactly: defining
\begin{equation}
\label{defb}
b(u,v)=(u\cdot\nabla_2) v-\left(\int_{-h}^z\diver_2u(\cdot,\cdot,\xi,\cdot)d\xi\right)\frac{\partial v}{\partial z},
\end{equation}
we may write the former as
\begin{equation}
\label{initial}
\frac{\partial v}{\partial t}-\Delta v +b(v,v)+\nabla_2 p=f.
\end{equation}

For any $j\in\mathbb{Z}$ we denote by $H^j(\mathbb{O})$ the Sobolev space of degree $j$ on $\mathbb{O}$ and
\begin{equation*}
	H^j:=\{v\in \left(H^j(\mathbb{O})\right)^2:\int_{\mathbb{O}}vd\zeta=0, v \text{ is even in }z\},
\end{equation*}
and provide it with the homogenous scalar product
\begin{equation*}
	\langle u,v\rangle_j=\langle (\nabla)^j u,(\nabla)^jv\rangle.
\end{equation*}
Next for $j\in \mathbb{N}\cup\{0\}$ we set
\begin{equation*}
	V^j:=\{w\in H^j: w \text{ satisfies }\eqref{0D}\}. 
\end{equation*}
We further note:
\begin{equation*}
	H^r\oplus\mathbb{R}\mathbbm{1}:=\{v(t)+c, v\in H^r, c\in\mathbb{R}\}.
\end{equation*}
Moreover, we abbreviate $H^0$ by $H$ and $V^0$ by $V$.
% and $\mathscr{D}(\mathbb{O}\times(0,T))$ being the set of infinitely continuously derivable functions with support on a compact domain of  $\mathbb{O}\times(0,T)$.
%, and
%\begin{equation*}
%    \Tilde{V}:=\left\{v\in\mathscr{D}(\mathbb{O}\times(0,T)), \int_{-h}^0\diver_2 v=0\right\},\quad||\cdot||_{V}:=||\cdot||_{L^2},
%\end{equation*}

For $j\in\mathbb{N}\cup 0$, let us define:
\begin{itemize}
	\item $E_j([t_1,t_2]):=L^2([t_1,t_2], V^j)$ %to which we bestow the norm $||\cdot||_j$ 
	%for $j\in\mathbb{N}\cup\{0\}$,
	\item $U_j([t_1,t_2]):=\{u\in E_{j+1}([t_1,t_2]),\frac{\partial u}{\partial t}\in E_{j-1}([t_1,t_2])\}$. Notice that $U_j([t_1,t_2])\hookrightarrow C([t_1,t_2],V^{j})$ , see Theorem 3.1 in the first chapter of \cite{BOK7}.
	%	\item moreover we write $|||f|||_k:=\sup_{\mathbb{R}^+} ||f||_{V^k}$ and $||f||_{k}:=||f||_{L^2([0,T],H^k)}$,
	%
	%	\item we furthermore set the following notations: $E_j(T):=E_j([0,T])$ and $U_j(T):=U_j([0,T])$,
	\item we abbreviate $E_j(T):=E_j([0,T])$ and $U_j(T):=U_j([0,T])$, and denote %$||f||_{k}:=||f||_{E_k(T))}$ and  
	$|||f|||_k:=\sup_{t} ||f(t)||_{V^k}$.

	\item lastly, for a Banach space $E$
	%for any normed space $E$, 
	we denote by $B_E(a,R)$ the open ball in $E$ centered in $a$ with the radius $R$ %with respect to the norm of $E$ 
	and $\bar{B}_E(a,R)$ the closed ball; we shall omit $a$ when it be zero.
	
%	\item $E_j:=E_j(T):=L^2([0,T], V^j)$ which we equip with $||\cdot||_j$ for $j\in\mathbb{N}\cup\{0\}$,
%	\item $U_j:=U_j(T):=\{u\in E_{j+1}(T),\frac{\partial u}{\partial t}\in E_{j-1}(T)\}$, we notice that $U_j\hookrightarrow C([0,T],V^{j})$ for $j\in\mathbb{N}\cup\{0\}$
%	\item and we expand those definitions to $E_j([t_1,t_2]):=L^2([t_1,t_2], V^j)$ and $U_j:=\{u\in E_{j+1}([t_1,t_2]),\frac{\partial u}{\partial t}\in E_{j-1}([t_1,t_2])\}$.
\end{itemize}

This done we may state, grounding ourselves on the results of \cite{ART9}, the existence and uniqueness of solutions of the primitive equations, which we define in the following way:
\begin{definition}
	\label{solution}
	Let $T>0$, $m\geq 1$ and $v_0\in V^m$. Then $v$ is a strong solution of $\eqref{init}, \eqref{ci}$ if  $v\in U_m(T)$ and $\exists p\in L^2([0,T], H^m(\mathbb{T}_L^2))$ such that $(v,p)$ verifies $\eqref{init}$ and $\eqref{ci}$.
	%%% La régularité de $\nabla p$ et donc celle de $p$ peuvent être établies avec l'équation et la régularité des autres termes.
\end{definition}
%%%% Original
%\begin{definition}
%Let $T>0$ and $v_0\in V^1$, $v$ is a strong solution of $\eqref{init}, \eqref{ci}$ if  $v\in C([0,T],V^1)\cap L^2([0,T],V^2(\mathbb{O})$, $\frac{dv}{dt}\in L^1([0,T],L^2(\mathbb{O}))$ and $\exists p\in L^1(\mathbb{O}\times(0,T))$ such that $(v,p)$ verifies $\eqref{init}$
%\end{definition}
Moreover, if $v(\zeta,t),t\geq 0$ is such that $v(\zeta,t)|_{0\leq t\leq T}$ is a solution of $\eqref{init}$, $\eqref{ci}$ for every $T>0$, then $v$ is called  a solution over $\mathbb{R}^+$. 

To establish the well-posedness of the problem $\eqref{init}$, $\eqref{ci}$ we refer to \cite{ART9}
% which allows us to directly state some high-order estimates—
(see also \cite{ART3} for a similar result). %in case of a constant in time force. Indeed by using 
Namely, Theorems 2.1 and  3.1 from \cite{ART9} imply:  
\begin{theorem}
	\label{Petcu}
	Let $m\geq 1$, $v_0\in V^m$ and $f\in L^{\infty}(\mathbb{R}^+,V^{m-1})$. Then problem $\eqref{init}$, $\eqref{ci}$ has a unique solution $v\in C(\mathbb{R}^+,V^m)\cap L^2_{loc}(\mathbb{R}^+,V^{m+1})$. Moreover, for any $T>0$, the norm of $v$ in $C(\mathbb{R}^+,V^m)\cap L^2_{loc}(\mathbb{R}^+,V^{m+1})$ depends only on $T$, $||v_0||_V^m$ and $|||f|||_m$.%$U_m$ on  $\mathbb{R}^+$.
	%% Pour une solution, le passage d'un espace à l'autre est direct
	% $C(\mathbb{R}^+,V^m)\cap L^2_{loc}(\mathbb{R}^+,V^{m+1})$
\end{theorem}
%By virtue of $\eqref{proj}$, an immediate corollary to this is the following:
%\begin{corollary}
%	Keeping the same setting as in theorem $\ref{Petcu}$ we further state that our solution is in $U_m(T)$ for all $T>0$.
%\end{corollary}

The deterministic problem thus settled we may now turn to the stochastic one, pondering in particular whether it is $\textit{mixing}$.
The primitive equations $\eqref{init}$, $\eqref{ci}$ would be said to be $\textit{mixing}$ if for any starting point $v_0\in V^m$, a solution $v(t)$ of $\eqref{init}$, $\eqref{ci}$ weakly converged to a certain measure $\mu$ in the space $V^m$. That is, if for every bounded and continuous functional $g$ on $V^m$ we had
\begin{equation}
\label{melsto}
\mathbb{E}[f(v(t)]\rightarrow \int_\Omega f(u)\mu(du),\quad \text{as } t\rightarrow\infty.
\end{equation}
%
%The mixing means that for any random starting point $v_0\in V^m$, the stochastic process generated by our system weakly converge toward law $\mu$, i.e
%\begin{equation}
%\label{melsto}
%\mathbb{E}[f(v_k(v_0))]\rightarrow \langle f,\mu\rangle,
%\end{equation}
%this being true for any continuous, bounded function $f$. All those notions will be defined later. We shall state this result in a more explicit form in section $\ref{verifhyp}$.
%This result is interesting in that it allows the statistical study of the system.
%Usually, we take two sources of randomness in our system: in the initial condition and in a external forcing component. 
%

For the related 2D Navier-Stokes system, the mixing is established for various classes of random forces, bounded and unbounded.
First off, in \cite{ART33} the mixing property was established for the 2D Navier-Stokes system under a bounded noise defined as a kick-force. 
Thereafter, the 2D Navier-Stokes was shown in \cite{ART29} to be mixing when subjected to a white-in-time noise under the condition that said noise be "rough" enough. 
%This latter condition was relaxed in \cite{ART32, ART34}, the noise being expressed only on a finite but large number of Fourier modes as a space variable.
Under the more relaxed condition that this noise be expressed only on a finite, but large number, of Fourier modes as a space variable, the uniqueness of a stationnary measure for the 2D Navier-Stokes problem was shown in \cite{ART26}, and then the mixing of this system in \cite{ART32, ART34}.
However, when the white-noise, to which is subjected the 2D Navier-Stokes system, is expressed on a small number of modes, only the existence and uniqueness of a stationnary measure, have been proved, namely in \cite{ART35bis}.
All the former results about the existence and uniqueness of a stationnary measure for the 2D Navier-Stokes equations under a white noise have been summed up in \cite{ART43}.
%Moreover it was shown in \cite{ART35}, that the 2D Navier-Stokes, perturbed by a white noise expressed on a small number of modes only, has a unique stationnary measure .

For the 3D Navier-Stokes system, on the toroidal layer $\mathbb{T}_L^2\times (0,\epsilon)$ (with suitable boundary conditions), the mixing was proved in \cite{ART17} for bounded kick-forces, when $\epsilon$ is small in terms of the force. 

Now, as regards the primitive equations $\eqref{init}$, $\eqref{ci}$, they were studied in \cite{ART3} in the case when $f$ is a bounded kick-force, and a sketch of the proof of mixing was given. Pertaining to the primitive equations system perturbated by a white-in-time force, the well-posedness of the problem was proved in \cite{ART13, ART14} (for two different forms of white noise) and the existence of a stationnary measure in \cite{ART15}.
However, for equations $\eqref{init}$, $\eqref{ci}$ with white in time random force $f$, it still is not proven that the second moment of high Sobolev norms of solutions remains bounded, in the sense that
\begin{equation*}
\mathbb{E}[||v(t)||_m^2]<\infty,\quad \forall t>0, \quad\text{ for some } m\geq 2.
\end{equation*}
Without such a result it seems impossible to prove the uniqueness of a stationary measure and mixing.

%the stating of estimates for the trajectories of the solution that would not depend on the randomness—thus deterministic ones— is a daunting task. 
%Therefore as proving any mixing property without this seems quite difficult, we circumvent this issue by study the primitive equations under bounded noise 
As we explained earlier, in order to avoid this difficulty we study the primitive equations stirred by red forces (or red noises), that is bounded random forces, statistically periodic in time given by non-degenerate random Haar series. The exact form of those forces will be given in section $\ref{rednoisesec}$.
Our case of coloured noises is rather harder than the one of the kick-forces, treated in \cite{ART3}, the latter being also less adequate for applications.  

Thanks to the form of the noise and by Theorem $\ref{Petcu}$ the problem $\eqref{init}$ , $\eqref{ci}$ is well posed.
%In this work we study the case where our system is stimulated by a random bounded noise
%— a notion explained for example in \cite{ART4}—
 
The main result of this work is the following proposition: 
%\begin{theolettre}
%	The stochastic process generated in $V^m$ by $\eqref{init}$ stirred by a coloured non-degenerate bounded noise $f$, is exponentially mixing in spaces $V^m$, for $m\geq 2$.
%	% c'est à dire que les chaînes de Markov ainsi créées admettent une mesure stationnaire unique et estimation de la vitesse de convergence?
%\end{theolettre}
\begin{theolettre}
	The stochastic process generated in $V^m$ by $\eqref{init}$ stirred by red-noise $f$ as above, evaluated at integer moments of time define a Markov chain which is exponentially mixing in spaces $V^m$, for $m\geq 2$.
	% c'est à dire que les chaînes de Markov ainsi créées admettent une mesure stationnaire unique et estimation de la vitesse de convergence?
\end{theolettre}

The first part of this work is devoted to stating some deterministic results about our system. In the second part we study its stochastic side, introducing the red noise before stating a theorem from \cite{ART7} which enables us to prove that our system is exponentially mixing.% and which furthermore gives us an estimate for the rate of convergence of the laws of the Markov chain toward the stationary measure. 
\section*{Acknowledgements}
First and foremost, I would like to thank Serguei Kuksin, for his dedication in overseeing this work. Furthermore my gratitude goes to Vahagn Nersesyan for his kindness and his very insightful remarks. Lastly, I thank Mrs Madalina Petcu for making herself available to answer my questions and provide advice.

Moreover, I am grateful to Université Paris Diderot, for its financial support in the form of a PhD grant, without which, I would never have written this work.

\section{Preliminary results}
Now let us introduce some useful results for the calculations, lying ahead.
\begin{formula}
\label{ortho}
For any $v\in V^1$ and $p\in H^1(\mathbb{T}_L^2)$%%%avant $p\in L^2(\mathbb{T}_L^2)$
    \begin{multline*}
    	\int_{\mathbb{O}}v\cdot \nabla_2 p=\int_{\mathbb{O}}v(x,y,z)\cdot\nabla_2 p(x,y)dxdydz=\int_{\mathbb{T}_L^2}\Bar{v}(x,y)\cdot\nabla_2p(x,y)dxdy
    	\\
    	=\int_{\mathbb{T}_L^2}\diver_2(\Bar{v})(x,y)p(x,y)dxdy=0.
    \end{multline*}
\end{formula}
Moreover,
\begin{lemma}
	\label{projL2}
	%We denote by $\tilde{L}_2(\mathbb{O})^2$ the space of functions in $L_2(\mathbb{O})^2$ that are even in $z$ and have zero mean value.
	We have the following orthogonal decomposition:
	\begin{equation*}
	H= V\oplus \nabla_2H^1(\mathbb{T}_L^2),
	\end{equation*}
	where %$\nabla_2H^1(\mathbb{T}_L^2):=\{y\in L^2, \exists p \in H^1(\mathbb{T}_L^2), y=\nabla_2 p \}$.
	$\nabla_2H^1(\mathbb{T}_L^2):=\{\nabla_2 p, p \in H^1(\mathbb{T}_L^2)\}$.
\end{lemma}
\begin{proof}
	%We denote here by $\tilde{L}_2(\mathbb{O})^2$ the space of vectors of $L_2(\mathbb{O})^2$ that are even in $z$ have zero mean value. 
	Clearly $\nabla_2 H^1(\mathbb{T}_L^2)$ belongs to $H$.
	Now let us consider the complex base of $H^1(\mathbb{T}_L^2)$ made of vectors $e^{2i\pi/L(mx+ny)}$, for $ m$ and $n$ in $\mathbb{Z}$.
	Then $\nabla_2 H^1(\mathbb{T}_L^2)$ is spanned by
	\begin{equation*}
		\Upsilon_{m,n}:= 
		\begin{pmatrix}
			m\\
			n
		\end{pmatrix}
		e^{(2i\pi/L)(mx+ny)}, \quad \forall m,n\in\mathbb{Z},\quad |m|+|n|\neq 0.
	\end{equation*}
	In view of Formula $\ref{ortho}$, the spaces $V$ and $\nabla_2H^1(\mathbb{T}_L^2)$ are orthogonal. So it remains to show that each vector, orthogonal to $\nabla_2H^1(\mathbb{T}_L^2)$ belongs to $V$. %To do this consider any $v\in \tilde{L}_2(\mathbb{O})^2$. 
	The space $H$ admits the orthogonal complex basis $\{\Lambda_{m,n}^{k\pm}:m,n\in\mathbb{Z},k\in\mathbb{N}\cup\{0\}$, %|m|+|n|>0\} $, 
	where 
	\begin{align*}
		\Lambda_{m,n}^{k +}&=
		\begin{pmatrix}
			m
			\\
			n
		\end{pmatrix}
		e^{(2i\pi/L)(mx+ny)}\cos(kz),
		\\
		\Lambda_{m,n}^{k -}&=
		\begin{pmatrix}
		-n
		\\
		m
		\end{pmatrix}
		e^{(2i\pi/L)(mx+ny)}\cos(kz).
	\end{align*}
	Decomposing any $v\in H$, such that $v\perp \nabla_2 H^1(\mathbb{T}_L^2)$ in this basis as
	\begin{equation*}
		v=\sum_{m,n,k}\left(v_{m,n}^{k+}\Lambda_{m,n}^{k+}+v_{m,n}^{k-}\Lambda_{m,n}^{k-}\right),
	\end{equation*}
	we see that
	\begin{equation*}
		0=\langle v,\Upsilon_{m,n}\rangle=  hL(m^2+n^2)v_{m,n}^{0,+}\cdot\delta_{k,0}.
	\end{equation*}
	Thus 
	\begin{equation*}
		v=\sum_{m,n,k\neq 0}v_{m,n}^{k\pm}\Lambda_{m,n}^{k\pm}+\sum_{m,n}v_{m,n}^{0-}\Lambda_{m,n}^{0-}.
	\end{equation*}
	The function $v$ satisfies $\eqref{cb2}$ and $\eqref{0M}$ since it belongs to $H$ 
	% $\tilde{L}_2(\mathbb{O})$ 
	and it satisfies $\eqref{0D}$ since the functions $\Lambda_{m,n}^{k-}$ and $\Lambda_{m,n}^{0\pm}$ obviously do. Therefore $v\in V$.
\end{proof}
We note that the basis $\{\Lambda_{m,n}^{k\pm}\}$ is an orthogonal basis of space any $H^j$ and we define the projection $\mathfrak{P}$ in the following way:
\begin{equation}
	\label{projection}	
	\mathfrak{P}:\sum_{m,n,k}\Big(v_{m,n}^{k+}\Lambda_{m,n}^{k+}+v_{m,n}^{k-}\Lambda_{m,n}^{k-}\Big)\mkern-5mu\rightarrow\mkern-5mu \sum_{m,n,k\neq 0}\mkern-5mu v_{m,n}^{k\pm}\Lambda_{m,n}^{k\pm}+\sum_{m,n}v_{m,n}^{0-}\Lambda_{m,n}^{0-}. 
\end{equation}
Accordingly, the collection of functions $\{\Lambda_{m,n}^{k\pm}, m,n\in \mathbb{Z},k\in\mathbb{N}\}\cup\{\Lambda_{m,n}^{0-}, m,n\in \mathbb{Z}\}$ %Ne pas oublier dans les naturels de rajouter 0
%(here as before $|m|+|n|\neq 0$) 
is an orthogonal basis of each space $V^j$. We denote by $(e_n)_{n\in\mathbb{N}}$ the orthonormal basis of $V$, obtained by the normalization of $\{\Lambda_{m,n}^{k\pm}, m,n\in \mathbb{Z},k\in\mathbb{N}\}\cup\{\Lambda_{m,n}^{0-}, m,n\in \mathbb{Z}\}$ with respect to the $L^2$-norm.
The functions $\Lambda_{m,n}^{k\pm}$ are eigenvectors for the operator $\Delta$. Indeed
\begin{equation*}
	-\Delta \Lambda_{m,n}^{k\pm}=\left((2\pi/L)^2(m^2+n^2)+h^2k^2\right)\Lambda_{m,n}^{k\pm}.
\end{equation*}
From this fact making use of the form $\eqref{projection}$ of the operator $\mathfrak{P}$, we immediately get :
\begin{lemma}
	\label{projVr}
	For any $r\in\mathbb{N}$ the operator $\mathfrak{P}:H^r\oplus\mathbb{R}\mathbbm{1}\mapsto V^r$ is an orthogonal projection.
%	For any $r\in\mathbb{N}$ the operator $\mathfrak{P}:H^r\mapsto V^r$ is an orthogonal projection.
%	Furthermore setting:
%	\begin{equation*}
%		H^r\oplus\mathbb{R}:=\{v(t)+c, v\in H^r, c\in\mathbb{R}\},
%	\end{equation*}
%	we have that $\mathfrak{P}:H^r\oplus\mathbb{R}\mapsto V^r$ is also an orthogonal projection.
\end{lemma}

Thanks to the eigenvalues of $-\Delta$ associated with the orthogonal basis $\Lambda_{m,n}^{k\pm}$ being all positive we get the Poincaré inequality.%Car dans le produit scalaire IPP et -\Delta d'un côté$
\begin{lemma}
		\label{Poinc}
		There exists a $C>0$ such that for any $v\in V^1$,
		\begin{equation}
		\label{Poincare}
		||v||_{L^2}\leq C||\nabla v||_{L^2}.
		%\tag{P}
		\end{equation}
\end{lemma}

Now, considering the non-linear term of equation $\eqref{init}$
%\begin{equation}
%	\label{defb}
%	b(u,v)=(u\cdot\nabla_2) v-\Big(\int_{-h}^z\diver_2(u)\Big)\frac{\partial v}{\partial z},
%\end{equation}
we get,
\begin{lemma}
	\label{intpart}
	For any $w\in V$, $v\in \left(H^2(\mathbb{O})\right)^2$%— periodic over our domain—
	\begin{equation*}
	\langle b(w,v),v\rangle=0.
	\end{equation*}
\end{lemma}
\begin{proof}
	Indeed, % replacing $w$ with a regular enough approximation we may write
	$\langle(w\cdot\nabla_2)v,v\rangle=%-\langle(w\cdot\nabla_2)v,v\rangle
		-\frac{1}{2}\langle \diver_2(w)v,v\rangle$
	and
	\begin{equation*}
	\left\langle\left (\int_{-h}^z \diver_2 w(\cdot,\cdot,\xi,\cdot)d\xi\right)\partial_z v,v\right\rangle=
	-\frac{1}{2}\langle \diver_2(w) v,v\rangle. %-\langle (\int_{-h}^z \diver_2(w))\partial_z v,v\rangle
	\end{equation*}
	The first formula is obtained through integration by part (over the two horizontal variables only):
	\begin{equation*}
	\int_{\mathbb{T}_L^2}((w\cdot\nabla_2)v)v=
	%\int_{\mathbb{T}_L^2}\diver_2(|v|^2)(w\cdot\vec{n})
	-\int_{\mathbb{T}_L^2}((w\cdot\nabla_2)v)v-\int_{\mathbb{T}_L^2}\diver_2(w)v^2.
	\end{equation*}
	%$v$ étant nul sur $\Gamma_s$ i.e sur les bords de $\mathbb{T}_L^2$.
	As for the second it is yielded by the integration by parts over $z$ compounded with $\int_{-h}^z \diver_2(w)|_{z=0}= 0$.
	Hence the result.
\end{proof}
 %A frequent case in latter developments will be $v=w$.
Moreover, we have the following result %has introduced in \cite{ART3} and \cite{ART9}
% We consider the non-linear term of equation $\eqref{init}$
%\begin{equation}
%	\label{defb}
%	b(u,v)=(u\cdot\nabla_2) v-\int_{-h}^z\diver_2(u)\frac{\partial v}{\partial z}
%\end{equation}
%%%% Ancienne version plus précise
%\begin{lemma}
%    \label{techb}
%    Setting $u=(u_1,u_2), v=(v_1,v_2)\in (L^2(\mathbb{O}))^2$, we may write $\nabla b(u,v)$ as a combination of  terms that are components of 
%    \begin{equation*}
%    D^ku_j\tilde{D}^{2-k}v_{j'},\quad \left(D^k\int_{-h}^z \diver_2u_j(\cdot,\cdot,\xi)d\xi\right)\tilde{D}^{1-k}\frac{\partial v_{j'}}{\partial t},\quad k=0,1,\quad j,j'=1,2 
%    %&DuDv 
%    \end{equation*}
%    where $D^l$ and $\tilde{D}^l$ are two monomial terms of degree $l$ made from the spatial derivatives ($\partial_x$,$\partial_y$,$\partial_z$).
%    Likewise, $\Delta b(u,v)$ may be out-laid as a sum of terms of the like of
%    \begin{equation*}% LE CAS DE LA DERIVEE DE LA PRIMITIVE EN Z SE RAPPORTE AU PREMIER
%    D^ku_j\tilde{D}^{3-k}v_{j'},\quad \left(D^k\int_{-h}^z\diver_2u_j(\cdot,\cdot,\xi)d\xi\right)\tilde{D}^{2-k}\frac{\partial v_{j'}}{\partial z},\quad k=0,1,2,\quad j,j'=1,2 .
%    %&D^{l+1}uD^{2-l}v & l=0,1
%    \end{equation*}
%\end{lemma}
\begin{lemma}
	\label{techb}
	Setting $u=(u_1,u_2), v=(v_1,v_2)\in (L^2(\mathbb{O}))^2$, we may write $\nabla b(u,v)$ as a combination of  terms 
	\begin{equation*}
	D^ku_jD^{2-k}v_{j'},\quad \left(D^k\int_{-h}^z \diver_2u_j(\cdot,\cdot,\xi)d\xi\right)D^{1-k}\frac{\partial v_{j'}}{\partial t},\quad k=0,1,\quad j,j'=1,2, 
	%&DuDv 
	\end{equation*}
	where $D^l$ is any monomial term of degree $l$ made from the spatial derivatives ($\partial_x$,$\partial_y$,$\partial_z$).
	Likewise, $\Delta b(u,v)$ may be out-laid as a sum of terms
	\begin{equation*}% LE CAS DE LA DERIVEE DE LA PRIMITIVE EN Z SE RAPPORTE AU PREMIER
	D^ku_jD^{3-k}v_{j'},\quad \left(D^k\int_{-h}^z\diver_2u_j(\cdot,\cdot,\xi)d\xi\right)D^{2-k}\frac{\partial v_{j'}}{\partial z},\quad k=0,1,2,\quad j,j'=1,2 .
	%&D^{l+1}uD^{2-l}v & l=0,1
	\end{equation*}
\end{lemma}
%Hence two important estimates: 
\begin{lemma}
	\label{estb}
	 We have the following estimates:
	\begin{align*}
	%||B(u,v)||_{L^2}\leq
	||b(u,v)||_{L^2}\leq & C\min(||u||_{V^3}||v||_{V^1},||u||_{V^1}||v||_{V^3}),
	\\
	%||B(u,v)||_{V^1}\leq
	||\nabla b(u,v)||_{L^2}\leq & C\min(||u||_{V^3}||v||_{V^2},||u||_{V^2}||v||_{V^3}).
	%	\\
	%	||\Delta b(u,v)||_{L^2}\leq &
	%	C||u||_{V^3}||v||_{V^3}
	\end{align*}
	%	and, more generally
	%	\begin{equation*}
	%	||\Delta^{m/2}b(u,v)||_{L^2}\leq C||u||_{V^m}||v||_{V^m}
	%	\end{equation*}

	% We have the following estimate
	% \begin{equation*}
	%     \big|\big|\int_{-h}^z \diver_2 w_1\frac{\partial w_2}{\partial z}\big|\big|_{L^2}\leq C||w_1||_{V^3}||w_2||_{V^1},\quad\forall w_1,w_2 \in V^1
	% \end{equation*}
	% ,
	% \begin{equation*}
	%     \big|\big|\int_{-h}^z \diver_2 w_1\frac{\partial w_2}{\partial z}\big|\big|_{H^1}\leq C||w_1||_{V^3}||w_2||_{V^2},\quad\forall w_1,w_2 \in V^2
	% \end{equation*}
	% ,
	% \begin{equation*}
	%     ||\Delta Y(w_1,w_2)||_{L^2}\leq ||w_1||_{V^3}||w_2||_{V^4}
	% \end{equation*}
	% and
	% \begin{equation*}
	%     ||\nabla((w_1\cdot\nabla_2)w_2)||_{L^2}\leq||w_1||_{V^3}||w_2||_{V^2}
	% \end{equation*}
	% ,
	% \begin{equation*}
	%     ||(w_1\cdot\nabla_2)w_2||_{L^2}\leq ||w_1||_{L^\infty}||\nabla w_2||_{L^2}\leq ||w_1||_{V^2}||w_2||_{V^1}
	% \end{equation*}
	% ,
	% \begin{equation*}
	%     ||\Delta((w_1\cdot\nabla_2)w_2)||_{L^2}\leq||w_1||_{V^3}||w_2||_{V^3}
	% \end{equation*}
	%    et
	%    \begin{equation*}
	%        ||\diver_2 w_1\frac{\partial w_2}{\partial z}||_{L^2}\leq %C||w_1||_{V^1}||w_2||_{V^1},\quad\forall w_1,w_2 \in V^1
	%    \end{equation*}
	%    ,
	%    \begin{equation*}
	%        ||\diver_2 w_1\frac{\partial w_2}{\partial z}||_{V^1}\leq %C||w_1||_{V^2}||w_2||_{V^2},\quad\forall w_1,w_2 \in V^2
	%    \end{equation*}
\end{lemma}
\begin{proof}
	%We shall only prove the results up to $m=2$. 
	For the first one, it is enough to write
	\begin{align*}
	|| (u\cdot\nabla_2)v ||_{L^2}&\leq ||u||_{L^\infty}||\nabla_2v||_{L^2} \qquad\leq||u||_{V^2}||v||_{V^1} 
	\\
	&\textbf{or } ||u||_{L^6}||\nabla_2v||_{L^3}\qquad\leq||u||_{V^1}||v||_{V^2},
	\\
	\left|\left| \int_{-h}^z\diver_2u\frac{\partial v}{\partial z}\right|\right|_{L^2}&\leq ||\diver_2 u||_{L^\infty}\left|\left|\frac{\partial v}{\partial z}\right|\right|_{L^2}\leq ||u||_{V^3}||v||_{V^1}
	\\
	&\textbf{or } ||\diver_2u||_{L^2}\left|\left|\frac{\partial v}{\partial z}\right|\right|_{L^\infty}\leq ||u||_{V^1}||v||_{V^3}.
	\end{align*}
	For the second, by virtue of lemma $\ref{techb}$, we only need to prove the following bounds
	\begin{align*}
	||D^kuD^{2-k}v||_{L^2}&\leq C ||u||_{L^\infty}||D^2 v||_{L^2}\quad\leq C ||u||_{V^2}||v||_{V^2}&& k=0,
	\\
	&\quad C||Du||_{L^6}||Dv||_{L^3}\quad\leq C||u||_{V^2}||v||_{V^2} &&k=1.
	\\
	\left|\left|D^k\int_{-h}^z\diver_2 u D^{1-k}\frac{\partial v}{\partial z}\right|\right|_{L^2}\mkern-15mu&\leq  C||\diver_2 u||_{L\infty}\left|\left|D\frac{\partial v}{\partial z}\right|\right|_{L^2}\mkern-15mu\leq C||u||_{V^3}||v||_{V^2}&& k=0, 
	\\
	&\textbf{or }C ||\diver_2 u||_{L^6}\left|\left|D\frac{\partial v}{\partial z}\right|\right|_{L^3}\mkern-15mu\leq C||u||_{V^2}||v||_{V^3}&& k=0,
	\\
	&\quad C||D \diver_2 u||_{L^6}\left|\left|\frac{\partial v}{\partial z}\right|\right|_{L^3}\mkern-15mu\leq C||u||_{V^3}||v||_{V^2}&& k=1, 
	\\
	&\textbf{or }C||D \diver_2 u||_{L^2}\left|\left|\frac{\partial v}{\partial z}\right|\right|_{L^\infty}\mkern-15mu\leq C||u||_{V^2}||v||_{V^3}&& k=1. 
	\end{align*}
\end{proof}
Lemmas $\ref{projL2}$ and $\ref{estb}$ allow us to rewrite $\eqref{initial}$ in a new form, seemingly weaker but actually equally strong.
Let $u,v\in C([0,T],V^2)\cap E_3$, then thanks to Lemma $\ref{estb}$, $b(u,v)\in L^2\left([0,T],\left(L^2(\mathbb{O})\right)^2\right)$. Moreover, $u$ and $v$ being even in $z$, so is $b(u,v)$. 
%Finally, the space of constant functions is orthogonal in $\left(H^r(\mathbb{O})\right)^2$ to $H^r$, for any $r\geq 0$. Thus $\mathfrak{P}$ can be extended into an orthogonal projection from $\{w\in \left(H^r(\mathbb{O})\right)^2, w \text{ is even in }z\}$ to $V^r$. 
%Taking here $r=0$, we define
Thus $b(u,v)\in H\oplus\mathbb{R}\mathbbm{1}$ and thanks to Lemma $\ref{projVr}$, applied to $r=0$, we define
\begin{equation*}
B(u,v):=\mathfrak{P}b(u,v),\quad u,v\in C([0,T],V^2)\cap E_3.
\end{equation*}
%we may write in our context that $\eqref{init}$ is equivalent to the following equation
Then for $m\geq 2$, $f\in E_{m-1}$ and $v_0\in V^m$, $\eqref{init}$, $\eqref{ci}$ is equivalent to:
\begin{equation}
\label{proj}
%\frac{\partial v}{\partial t}-\Delta v +\mathfrak{P}\Big((v\cdot\nabla_2)v-\Big(\int_{-h}^z \diver_2(v)\Big)\frac{\partial v}{\partial z}\Big)=f.
\frac{\partial v}{\partial t}-\Delta v +B(v,v)=f.
\end{equation}
supplemented with $\eqref{ci}$.
%where $\mathfrak{P}$ is the projection on $V$. 
Indeed we may consider $v$ a solution of $\eqref{init}$ and apply $\mathfrak{P}$ to $\eqref{init}$. Then we use Formula $\ref{ortho}$ to prove that $\mathfrak{P}(\nabla_2 p)=0$ and $\eqref{projection}$ to show that for $v\in V^r$, $\mathfrak{P}\Delta v=\Delta v$. We thus get $\eqref{proj}$.
%Indeed, if $v$ satisfies $\eqref{init}$, then applying $\mathfrak{P}$ to $\eqref{init}$ and using that $\mathfrak{P}(\nabla_2 p)=0$ by Formula $\ref{ortho}$, and that  we get $\eqref{proj}$.  Moreover we shall prove later that for any $m\geq 2$ and $v\in V^m$, we have $\Delta v \in V^{m-2}$ thus $\Delta v$ is left unchanged by $\mathfrak{P}$.
Now let $v\in U_m(T)$ be a solution of $\eqref{proj}$; we reconstruct $p$ thanks to the relation
\begin{equation*}
\nabla_2 p= b(v,v)-B(v,v),
\end{equation*} 
and then, $(v,p)$ is a solution of $\eqref{init}$.
Moreover using Lemma $\ref{projVr}$ we have the following corollary to Lemma $\ref{estb}$
\begin{corollary}
	\label{estbis}
	 We have the following estimates:
	\begin{align*}
	||B(u,v)||_{L^2}\leq||b(u,v)||_{L^2}\leq & C\min(||u||_{V^3}||v||_{V^1},||u||_{V^1}||v||_{V^3}),
	\\
	||B(u,v)||_{V^1}\leq||\nabla b(u,v)||_{L^2}\leq & C\min(||u||_{V^3}||v||_{V^2},||u||_{V^2}||v||_{V^3}).
	\end{align*}
\end{corollary}
Besides
\begin{lemma}
	\label{estbhaut}
	The bilinear component $B$ abides by the estimates
	\begin{equation*}
	||B(u,v)||_{V^m}\leq C||u||_{V^{m+1}}||v||_{V^{m+1}},\quad m\geq 2
	\end{equation*}
\end{lemma}
\begin{proof}
	For $m\geq 2$, the space $H^m$is an algebra; thus, firstly
	\begin{equation*}
	||(u\cdot\nabla_2)v||_{H^m}\leq C||u||_{V^m}||Dv||_{H^m}\leq C ||u||_{V^m}||v||_{V^{m+1}}
	\end{equation*}
	secondly
	\begin{equation*}
		\left|\left|\int_{-h}^z \diver_2 u\frac{\partial v}{\partial z}\right|\right|_{H^m}\leq C||D u||_{H^m}||Dv||_{V^m}\leq C||u||_{V^{m+1}}||v||_{V^{m+1}}.
	\end{equation*}
\end{proof}
Thus by virtue of $\eqref{proj}$, an immediate corollary to Theorem $\ref{Petcu}$ is the following:
\begin{corollary}
	\label{corespace}
	For $m\geq 2$, under the assumptions of Theorem $\ref{Petcu}$ we further state that the solution $v$ belongs to the space $U_m(T)$ for all $T>0$.
\end{corollary}
\subsection{Absorbing set for the system}

First here we shall state the existence of an absorbing set for our system as written in \cite{ART9} . This result, one proves by using the a priori estimates, which through Gronwall's lemma yield long term boundedness of the solutions of $\eqref{init}$ (see \cite{ART3} for more details for the case of a constant force):
%— will be here written in the following way:
\begin{theorem}
	\label{Absorbe}
	For an integer $m\geq 1$, suppose that there is a constant $C^*$ such that
	\begin{equation*}
	||f(t)||_{V^{m-1}}\leq C^*,\quad\forall t\geq 0
	\end{equation*}
	then there exists $R=R(C^*)$ such that for all $M>0$, there exist $T=T(M,C^*)$ and $K=K(M,C^*)\geq 2R$
	for which a solution of $\eqref{init}$ with initial value $||v_0||_{V^m}\leq M$ verifies:
	\[
	\begin{cases}
	||v(t)||_{V^m}\leq K,& t\geq 0,
	\\
	||v(t)||_{V^m}\leq R,&  t\geq T
	\end{cases}
	\]
\end{theorem}
Without loss of generality in the following we assume that $T\geq 1$ and $R\geq 1$. Now let us consider the regularity of the resolving operator for the problem $\eqref{init}$, $\eqref{ci}$.
\section{Analyticity of the resolving operator}
Let us define the following sets and operators:
\begin{itemize}
%	\item $E_j:=E_j(T):=L^2([0,T], V^j)$ which we equip with $||\cdot||_j$ for $j\in\mathbb{N}\cup\{0\}$
%	\item $U_j:=U_j(T):=\{u\in E_{j+1}(T),\frac{\partial u}{\partial t}\in E_{j-1}(T)\}$, we notice that $U_j\hookrightarrow C([0,T],V^{j})$ for $j\in\mathbb{N}\cup\{0\}$
	\item for $T\leq\infty$, $\mathcal{E}_j(T):=\bar{B}_{L^{\infty}([0,T],V^j)}(C^*)$ for $j\in\mathbb{N}\cup\{0\}$, where $C^*$ is the same as in Theorem $\ref{Absorbe}$;
	\item for $0<T<\infty$ and some $\epsilon>0$ which we shall fix later,  we set 
	\begin{equation*}
	\hat{E}_{j}(T,\epsilon):=\mathcal{E}_{j}(T)+B_{E_{j}(T)}(\epsilon)\subset E_{j}(T);
	\end{equation*}
%	and define the following subset of $V^m$:
%	\begin{equation*}
%	\mathcal{O}_m:=\{v=Sol(v_0,\eta)(t), 0\leq t\leq 2T ,v_0\in B_{V^m}(R), \eta\in  \hat{E}_{m-1}(2T)\}
%	\end{equation*} 
	\item $\mathcal{L}$ is the operator describing equation $\eqref{proj}$, $\eqref{ci}$:
	\begin{equation*}
	\mathcal{L}:v \mapsto (\mathcal{L}^1v:=v(0),\mathcal{L}^2v:=\frac{\partial v}{\partial t}-\Delta v+B(v,v)),
	\end{equation*}
	\item $Sol: (u_0,\eta)\mapsto u$ where $u$ is a solution of $\eqref{proj}$, $\eqref{ci}$, that is $Sol$ is the inverse of $\mathcal{L}$,
	\item $d\mathcal{L}$ the linearisation of $\mathcal{L}$:%linked to the linearised equation by $\eqref{init}$
	\begin{equation*}
	d\mathcal{L}(v): w\mapsto (d\mathcal{L}^1w:=w(0),d\mathcal{L}^2w:=\frac{\partial w}{\partial t}-\Delta w+B(v,w)+B(w,v)).
	\end{equation*}
	%where $b(v,w):=(v\cdot\nabla_2)w-\int_{-h}^z\diver_2v\frac{\partial w}{\partial z}$.
	%\item $S:(u_0,\eta)\rightarrow u(1)$ which is coherent with the above given definition of $S$ as a transition operator.
\end{itemize}
Now, by virtue of Theorem $\ref{Petcu}$, for all $m\geq 1$, $M>0$, $0\leq T'\leq T$ and $(v'_0,f')\in B_{V^m}(M)\times \mathcal{E}_{m-1}(T')$, $Sol(v'_0,f')(t)\in V^m$ is defined for $t\in[0,T']$. In this setting, we have the following result
%%%% Ancienne version qui marche aussi
%\begin{theorem}
%	\label{AnalSol}
%	For any $T>0$, $M>0$, and $m\geq 2$, there exists $\rho=\rho(M,T,m)>0$ and $R=R(M,T,m)>0$, such that for any $T'$ such that $\frac{T}{2}\leq T'$, 
%	%%%Ancienne version point à point
%	%and for all $(v'_0,f')\in B_{V^m}(M)\times \mathcal{E}_{m-1}(T')$, operator $Sol$ is well defined over $B_{V^m}(v'_0,\epsilon)\times B_{E_{m-1}(T')}(f',\epsilon)$ and is analytic over this set into $U_{m+1}$.
%	$Sol$ is an analytical mapping from $B_{V^m}(M)\times \hat{E}_{m-1}(T',\rho)$ into $U_{m+1}$.
%	%
%	%
%	%%%% Propriété supplémentaire pourquoi vire t-on ça va avec formulation précédente c'est juste que Kuskin n'a pas compris?
%	%More particularly, reusing $R$ as introduced in theorem $\ref{Absorbe}$ and setting $M=R$, we then get that $Sol$ is analytical from $B_{V^m}(R)\times \hat{E}_{m-1}(T',\epsilon)$ into $U_{m+1}$.
%	%%%% Local Lipschtziannity
%	Besides
%	%%% Formulation précédente
%	%	\begin{equation*}
%	%	||Sol(v_0,f)(t)-Sol(v'_0,f')||_{V^m}\leq 1
%	%	\end{equation*}
%	\begin{equation*}
%	||Sol(v_0,f)(t)-Sol(v'_0,f')||_{E_m}\leq 1,\quad \forall (v_0,f_0)\in B_{V^m}(v'_0,\rho)\times B_{E_{m-1}(T')}(f',\rho)
%	\end{equation*}
%\end{theorem}
\begin{theorem}
	\label{AnalSol}
	For any $T>0$, $M>0$ and an integer $m\geq 2$, there exists $\rho(M,T,m)>0$ such that for any $T'\in[\min(1,T/2),T]$, the mapping 
	\begin{equation}
	\label{eqSol}
	Sol:B_{V^m}(M)\times\hat{E}_{m-1}(T',\rho)\mapsto U_m(T'),
	\end{equation}
	is well defined and analytical.
	Besides it is Lipschitz with a Lipschitz constant $L$ which depends only on $M$, $T$ and $m$.
\end{theorem}
The proof of this theorem will be brought at the end of this section with the help of Theorem $\ref{invlocalunif}$ and thanks to some technical results that we are now going to lay out.
\subsection{Properties of operator $\mathcal{L}$}
First let us establish the set on which operator $\mathcal{L}$ is well-defined. To this end we begin by stating a corollary to Lemmas $\ref{estbis}$ and $\ref{estbhaut}$:

\begin{corollary}
	\label{cestb}
	We have
	\begin{equation*}
	||B(u,v)||_{E_1}\leq C ||u||_{U_2}||v||_{U_2},
	\end{equation*}
	and for $m\geq 2$,
	\begin{equation*}
	||B(u,v)||_{E_m}\leq C\min(||u||_{U_{m+1}}||v||_{U_m},||v||_{U_{m+1}}||u||_{U_m}).
	\end{equation*}
\end{corollary}
\begin{proof}
	The first inequality immediately follows from Corollary $\ref{estbis}$
	%is grounded on $\ref{estbis}$, which second inequality we take to the square and then integrate:
	\begin{equation*}
	\int_0^T||B(u,v)||_{V^1}^2\leq \int_0^TC||u||_{V^2}^2||v||_{V^3}^2\leq C||u||_{C([0,T],V^2)}^2||v||_{E_3}^2\leq C||u||_{U_2}^2||v||_{U_2}^2.
	\end{equation*}
	To get the second one, we proceed likewise, this time starting from Lemma $\ref{estbhaut}$.
	%For our second inequality we proceed likewise this time starting from $\ref{estbhaut}$.
\end{proof}
%With the help of those three results we may prove the following
From this corollary we derive the following
\begin{lemma}
	\label{bilin}
%	First let us state the range of $\mathcal{L}$:
%	\begin{equation*}
%		\mathcal{L}(\cdot):U_2\rightarrow E_{1}\times V^2
%	\end{equation*}
%	and more generally
	The operator
	\begin{equation*}
		\mathcal{L}:U_m\rightarrow V^m\times E_{m-1},\quad m\geq 2.
	\end{equation*}
	 is analytical.
\end{lemma}
%To prove this result let us first study of the bilinear component $b$ of the primitive equations
%
%This said we may now conclude the study of the definition space and range of $\mathcal{L}$
\begin{proof}%{of lemma $\ref{bilin}$}
	We analyze $\mathcal{L}$ term-wise. %component-wise.
%		Applying to $u=v$ lemmas $\ref{estbis}$ et $\ref{estbhaut}$, we get at worst:
%		\begin{equation*}
%		||b(u,u)||_{V^{m-1}}\leq ||u||_{V^{m+1}}||u||_{V^m}
%		\end{equation*}
%		So that
%		\begin{equation*}
%		\int_0^T||b(u,u)||_{V^{m-1}}^2\leq \int_0^T|u||_{V^{m+1}}^2||u||_{V^m}^2\leq ||u||_{C([0,T],V^m)}^2||u||_{E_{m+1}}^2\leq ||u||_{U^m}^4
%		\end{equation*}
%		for $U_{m}\subset C([0,T],V^m)$.
		By virtue of Corollary $\ref{cestb}$, we have that
		\begin{equation*}
			\int_0^T||B(u,u)||_{V^{m-1}}^2\leq ||u||_{U^m}^4.
		\end{equation*}
		Thus $B$ is an homogeneous operator of rank $2$, continuous and thus analytic. %, a continuous one thus an analytical one.
		%For the other components we get that
		Another term of $\mathcal{L}^2$, the operator
		\begin{equation*}
		u\mapsto \frac{d u}{dt}-\Delta u,\quad U_m\rightarrow E_{m-1},
		\end{equation*}
		obviously is linear continuous.
		Finally, the mapping $\mathcal{L}^1$
%		\begin{equation*}
%		u\mapsto u(0), \quad U_m\rightarrow V^m
%		\end{equation*}
		is linear and continuous since $U_{m}\subset C([0,T],V^m)$. Therefore, $\mathcal{L}$ is analytical.
%	First let  us establish the spaces on which $\mathcal{L}$ exists.
%	 By applying to our particular $u=v$ the lemmas $\ref{estbis}$ and $\ref{estbhaut}$, we get at worst:
%	\begin{equation*}
%	||b(u,u)||_{V^{m-1}}\leq ||u||_{V^{m+1}}||u||_{V^{m}},
%	\end{equation*}
%	For the other terms the result is immediate as for the initial term the results arises from $U_m\subset C([0,T],H^{m})$.
%	Now let us deal with the analycity; any potential difficulty here would lie in the analycity of the $b$ component of $\mathcal{L}$, lthe others being linear, bounded.
%	For the former as we know that
%	\begin{equation*}
%	||b(u,u)||_{V^{m-1}}\leq ||u||_{V^{m+1}}||u||_{V^m},
%	\end{equation*}
%	we have by studying the time $L^2$ norms of the terms of this inequality
%	\begin{equation*}
%	||b(u,u)||_{E^{m-1}}\leq ||u||_{U_m}||u||_{C([0,T], V^m)}\leq ||u||_{U_m}^2,
%	\end{equation*}
%	thus, $b$ being bilineard, it is analytical.
\end{proof}
%Moreover, we have the following result
%\begin{lemma}
%	\label{AnalL}
%	$\mathcal{L}$ is an analytical operator
%\end{lemma}
%\begin{proof}
%	Any potential difficulty here would lie in the analycity of the $b$ component of $\mathcal{L}$, lthe others being linear, bounded.
%	For the former as we know that
%	\begin{equation*}
%	||b(u,u)||_{V^{m-1}}\leq ||u||_{V^{m+1}}||u||_{V^m},
%	\end{equation*}
%	we have by studying the time $L^2$ norms of the terms of this inequality
%	\begin{equation*}
%	||b(u,u)||_{E^{m-1}}\leq ||u||_{U_m}||u||_{C([0,T], V^m)}\leq ||u||_{U_m}^2,
%	\end{equation*}
%	thus, $b$ being bilineard, it is analytical.
%\end{proof}
Then let us state an injection property for $\mathcal{L}$.
\begin{lemma}
	\label{injecL}
	Let $m\geq 2$, $T>0$ and $u_1,u_2\in U_m(T)$ such that $\mathcal{L}(u_1)=\mathcal{L}(u_2)$. Then $u_1(t)=u_2(t)$ for any $t\in[0,T]$.
\end{lemma} 
\begin{proof}
	If $\mathcal{L}(u_1)=\mathcal{L}(u_2)$,  then $\tilde{u}=u_1-u_2$ is a solution of
	\begin{equation*}
	\frac{d\tilde{u}}{dt}-\Delta\tilde{u}+B(u_1,\tilde{u})+B(\tilde{u},u_2)=0,\quad\tilde{u}_0=0.
	\end{equation*}
	Thus applying to the equation the scalar product with $\tilde{u}$, we get, thanks to Lemma $\ref{intpart}$ and Corollary $\ref{estbis}$,
	\begin{equation*}
	\frac{d||\tilde{u}||_{L^2}^2}{dt}+||\nabla\tilde{u}||_{L^2}^2\leq C||\tilde{u}||_{V^1}||u_2||_{V^3}||\tilde{u}||_{L^2}.% pas optimal Cao et Titi font mieux
	\end{equation*}
	Therefore, by virtue of Young's inequality, $\frac{d||\tilde{u}||_{L^2}^2}{dt}\leq C||u_2||_{V^3}^2||\tilde{u}||_{L^2}^2$. Whence
	\begin{equation*}
	||\tilde{u}(t)||_{L^2}^2\leq ||\tilde{u}_0||_{L^2}^2e^K,
	\end{equation*}
	so that $\tilde{u}(t)=0$ for any $t\in[0,T]$.
\end{proof}

Now let us examine the non-degeneracy of $\mathcal{L}$
\begin{lemma}
	\label{invL}
	For any $T>0$, $u\in U_2(T)$ and $r=0,1,2$, the mapping $d\mathcal{L}(u)$ is an isomorphism between $U_{r}(T)$ and $V^{r}\times E_{r-1}(T)$, whose inverse has a norm depending solely on $||u||_{U_2}$.
\end{lemma}
\begin{proof}
	To prove the existence of a solution to the linearised equation, i.e an inverse to $d\mathcal{L}$, and the boundedness of this inverse, let us start with preliminary estimates on the linearised equation. 
	In what follows, for $r=0,1$ we shall content ourselves with stating the a priori estimates, proving the full result for $r=2$ alone.
	
	% D'accord établissons quelques résultats utiles.
	% \begin{equation*}
	%     \int_{\mathbb{O}}\int_{-h}^z\diver_2 u\frac{\partial v}{\partial z}w\leq\int_{\mathbb{O}}\int_{-h}^z|\diver_2 u|\Big|\frac{\partial v}{\partial z}\Big||w|\leq h ||\diver_2 u||_{L^6}||\nabla v||_{L^2}||w||_{L^3}
	% \end{equation*}
	% et
	% \begin{multline*}
	%     \int_{\mathbb{O}}\int_{-h}^z\diver_2 v\frac{\partial u}{\partial z}w\leq\int_{\mathbb{O}}\int_{-h}^z|\diver_2 v|\Big|\frac{\partial u}{\partial z}\Big||w|\leq ||\frac{\partial u}{\partial z}||_{L^6}||\int_{-h}^0 \diver_2 v||_{L^2}||w||_{L^3}
	%     \\
	%     \leq ||\frac{\partial u}{\partial z}||_{L^\infty}\Big(\int_{-h}^0\int_{\mathbb{O}}|\diver_2|^2\Big)^{1/2}||w||_{L^2}\leq \sqrt{h}||\nabla u||_{L^6}||\nabla v||_{L^2}||w||_{L^3}
	% \end{multline*}
	% de même
	% \begin{equation*}
	%     \langle (u\cdot\nabla_2)v,w\rangle\leq\int_{\mathbb{O}}|u||\nabla v||w|\leq ||u||_{L^6}||\nabla v||_{L^2}||w||_{L^3}
	% \end{equation*}
	% et
	% \begin{equation*}
	%     \langle (v\cdot\nabla_2)u,w\rangle\leq\int_{\mathbb{O}}|v||\nabla u||w|\leq ||v||_{L^2}||\nabla u||_{L^3}||w||_{L^6}
	% \end{equation*}
	For the following considerations, we take any $(v_0,f)\in V^r\times E_{r-1}$ and denote $v=Sol(v_0,f)$, i.e $v$ is a solution of
	\begin{align}
		\label{initlin}
		\frac{d v}{dt}-\Delta v+B(u,v)+B(v,u)&=f,
		\\
		\label{cilin}
		v(0)&=v_0.
	\end{align}
	\textit{A priori estimates for the case $r=0$:}
	applying to $\eqref{initlin}$ the scalar product with $v$ in $L^2$ and using Lemma $\ref{intpart}$ and Corollary $\ref{estbis}$, and Cauchy-Schwartz inequality, we get
	\begin{equation*}
	\frac{1}{2}\frac{d ||v||_{L^2}^2}{dt}+||\nabla v||_{L^2}^2\leq ||f||_{V^{-1}}||v||_{V^1} +||B(v,u)||_{L^2}||v||_{L^2}.
	%\\
	%\leq 4||f||_{L^2}^2+C||u||_{V^3}||v||_{V^1}||v||_{L^2}
	\end{equation*}
	Henceforth we shall replace all constants by polyvalents $C$ and $C'$. Therefore
	\begin{equation*}
		\frac{d ||v||_{L^2}^2}{dt}+C||\nabla v||_{L^2}^2\leq C'(||f||_{V^{-1}}||v||_{V^{1}}+||u||_{V^3}||v||_{V^1}||v||_{L^2}).
	\end{equation*}
	Thus, applying once more Young's inequality to the last right-hand term, we get
	\begin{equation*}
		\label{estlinL2}
		\frac{d ||v||_{L^2}^2}{dt}+C||\nabla v||_{L^2}^2\leq C'(||f||_{V^{-1}}^2+||u||_{V^3}^2||v||_{L^2}^2).
	\end{equation*}
	Therefore, for any $0\leq t\leq T$,
	\begin{equation*}
	||v||_{L^2}^2(t)\leq (||v(0)||_{L^2}+C'||f||_{E_{-1}}^2)e^{C'\int_0^t||u||_{V^3}^2},
	\end{equation*}
	and thus
	\begin{equation}
	\label{L2lin}
	||v||_{L^2}^2(t)\mkern-3mu \leq\mkern-3mu (||v(0)||_{L^2}+C'||f||_{E_{-1}}^2)\chi(||u||_{U_2}),\mkern-5mu \text{ where } \chi(||u||_{U_2})\mkern-3mu := e^{C'||u||_{U_2}^2}.
	\end{equation}
	Moreover
	\begin{multline*}
	\int_0^t ||\nabla v||_{L^2}^2\leq \frac{1}{C}\left(||v(0)||_{L^2}+C'\int_0^t||f||_{V^{-1}}^2(s)+||u||_{V^3}^2(s)||v||_{L^2}^2(s)ds\right)
	\\
	\leq \frac{1}{C}\left(||v(0)||_{L^2}+C'||f||_{E_{-1}}^2+\int_0^t||u||_{V^3}^2(s)ds||v||_{\mathcal{E_0}}^2\right),
	\end{multline*}
	thus by virtue of $\eqref{L2lin}$, we get
	\begin{multline}
		\int_0^t ||\nabla v||_{L^2}^2\leq\left(||v(0)||_{L^2}+C'||f||_{E_{-1}}^2\right)\chi'(||u||_{U_2})
		\\
		\text{ with }\chi'(||u||_{U_2}):= \frac{1}{C}\left(1+||u||_{U_2}^2e^{C'||u||_{U_2}^2}\right).
	\end{multline}
	%%% VERSION D'AVANT QUI MARCHE AUSSI
%	\begin{equation*}
%	\frac{d ||v||_{L^2}^2}{dt}+||\nabla v||_{L^2}^2\leq |\langle f,v\rangle| +|\langle b(u,v)+b(v,u),v\rangle|
%	\end{equation*}
%	besides
%	\begin{multline*}
%	|\langle b(u,v)+b(v,u),v\rangle|\leq C(||u||_{L^6}||\nabla v||_{L^2}||v||_{L^3}+||\diver_2 u||_{L^6}||\nabla v||_{L^2}||v||_{L^3}
%	\\
%	+||v||_{L^2}||\nabla u||_{L^6}||v||_{L^3}+ ||\nabla u||_{L^6}||\diver_2 v||_{L^2}||v||_{L^3})
%	\\
%	\leq C(||u||_{V^2}||v||_{L^2}^{1/2}||\nabla v||_{L^2}^{3/2})
%	\end{multline*}
%	therefore,
%	\begin{equation}
%	\label{estlinL2}
%	\frac{d||v||_{L^2}^2}{dt}+\frac{||\nabla v||_{L^2}^2}{2} \leq ||f||_{L^2}||v||_{L^2}+ C(||u||_{V^2}||v||_{L^2}^{1/2})^4
%	\end{equation}
%	thus
%	\begin{equation}
%	\label{L2lin}
%	||v||_{L^2}^2(t)\leq\chi(t):=(||v(0)||_{L^2}+||f||_0^2)e^{\int_0^t||u||_{V^2}^4+1}
%	\end{equation}
	\textit{A priori estimates for the case $r=1$:}
	applying to  $\eqref{initlin}$ the scalar product with $-\Delta v$, Cauchy-Schwartz and Young's inequality to the first right-hand term as well as Corollary $\ref{estbis}$ and using once more Young's inequality on the second right-hand term, we get
	\begin{equation*}
		\frac{d||\nabla v||_{L^2}^2}{dt}+C||\Delta v||_{L^2}^2\leq C'(||f||_{L^2}^2+ ||u||_{V^3}^2||\nabla v||_{L^2}^2).
	\end{equation*}
%	thus applying Young's inequality to the last right-hand term, 
%	\begin{equation}
%		\label{estlinH1}
%		\frac{d||\nabla v||_{L^2}^2}{dt}+\frac{||\Delta v||_{L^2}^2}{2}\leq C(||f||_{L^2}^2+ ||u||_{V^3}^2||\nabla v||_{L^2}^2)
%	\end{equation}
	Therefore in a similar as for $r=0$,
	\begin{equation}
		\label{nabvlin}
		||\nabla v||_{L^2}^2(t)%\leq (C'||f||_{E_0}^2+||v(0)||_{V^1}^2)e^{\int_0^t C' ||u||_{V^3}^2}
		\leq (C'||f||_{E_0}^2+||v(0)||_{V^1}^2)\chi(||u||_{U_2}),
	\end{equation}
	and
	\begin{equation}
		\int_0^t||\Delta v||_{L^2}^2%\leq ||v(0)||_{V^1}+C\int_0^t ||f||_{L^2}^2+||u||_{V^3}^2||\nabla v||_{V^2}^2
		\leq\left(||v(0)||_{V^1}^2+C'||f||_{E_0}^2\right)\chi'(||u||_{U_2}). 
	\end{equation}

	\textit{Complete proof for the case $r=2$:}
	now let's apply to $\eqref{initlin}$ the scalar product with $-\Delta^2 v$, yielding
	\begin{equation*}
	\frac{1}{2}\frac{d||\Delta v||_{L^2}^2}{dt}+||\nabla\Delta v||_{L^2}^2\leq |\langle \nabla f,\nabla\Delta v\rangle|+|\langle \nabla(B(u,v)+B(v,u)),\nabla\Delta v\rangle|.
	\end{equation*}

	Thus by reusing Corollary $\ref{estbis}$, Cauchy-Schwartz and Young's inequalities, we get
	\begin{equation}
	\label{estlinH2}
	\frac{d||\Delta v||_{L^2}^2}{dt}+C||\nabla\Delta v||_{L^2}^2\leq C'(||f||_{V^1}^2+||u||_{V^3}^2||v||_{V^2}^2).
	\end{equation}
	Therefore, proceeding as we did for $r=0$, we get
	\begin{equation}
	\label{detvlin}
	||\Delta v||_{L^2}^2(t)%\leq (C'||f||_{E_1}^2+||v(0)||_{V^2}^2)e^{C'\int_0^t ||u||_{V^3}^2}
	\leq (C'||f||_{E_1}^2+||v(0)||_{V^2}^2)\chi(||u||_{U_2}),
	\end{equation}
	and
	\begin{equation}
	\int_0^t||\nabla\Delta v||_{L^2}^2%\leq\frac{1}{C}\left(||v(0)||_{V^2}^2+C'\int_0^t ||f||_{V^1}^2+||v||_{V^2}^2||u||_{V^3}^2dt\right)
	\leq \left(||v(0)||_{V^2}^2+C'||f||_{E_1}^2\right)\chi'(||u||_{U_2}).
	\end{equation} 
	
%	Thanks to all those estimates we may state (through Galerkin's approximations) the existence and, since $\eqref{initlin}$ is a linear equation also the uniqueness of a solution to $\eqref{initlin}$, and that it belongs to $C([0,T],V^2)\cap L^2([0,T],V^3)$.
		Thanks to all those estimates we may state (through Galerkin's approximations) the existence  of a solution for the problem $\eqref{initlin}$, $\eqref{cilin}$ and that it belongs to the space $C([0,T],V^2)\cap E_3(T)$. Since the equation is linear, the estimates also implies the uniqueness of a solution.
%	\begin{equation*}
%	 V^{2}\times E_1\rightarrow C([0,T],V^2)\cap L^2([0,T],V^3)
%	\end{equation*}
	%Hence, %, as $u\in L^2([0,T],V^3)$
	By virtue of Corollary $\ref{cestb}$, we have 
	
%	and as, by virtue of $\ref{estbis}$,
%	\begin{equation*}
%	||b(u,v)||_{E_1}+||b(v,u)||_{E_1}\leq ||v||_{C([0,T],H^2)}||u||_{E_3}
%	\end{equation*}
%	we have
	\begin{equation*}
	B(u,v)+B(v,u)\in E_1(T).
	\end{equation*}
	Now, since $\Delta v\in L^2([0,T],V^1)$, then in view of $\eqref{initlin}$, we have $\frac{\partial v}{\partial t}\in L^2([0,T],V^1)$. % in view of $\eqref{initlin}$,
	Therefore $v\in U_2(T)$.
	Thus the existence of an inverse operator is proven. Furthermore, we get from this argument that the norm of this operator is bounded by a value depending only on $||u||_{U_2}$. Thus our assertion is proved.	
\end{proof}
Naturally, a better smoothness of the curve $u(t)$ implies that the mapping $d\mathcal{L}(u)$ operates in smoother spaces:
\begin{lemma}
	\label{invLbis}
	For $m\geq 3$ and $u\in U_{m-1}$, the operator $d\mathcal{L}(u)$ is an isomorphism between $U_{m}$ and $V^{m}\times E_{m-1}$. Its inverse has a norm depending solely on $||u||_{U_{m-1}}$.
\end{lemma}
\begin{proof}
	We start by improving the regularity of the solution $v$ constructed in Lemma $\ref{invL}$.
	
	Applying to $\eqref{initlin}$ the scalar product with $(-\Delta)^m v$, we get the following estimate
	\begin{equation*}
	\frac{1}{2}\frac{d||\nabla^m v||_{L^2}^2}{dt}+||\nabla^{m+1} v||_{L^2}^2\leq |\langle B(u,v)+B(v,u),\Delta^{m} v\rangle|+|\langle f,\Delta^{m} v\rangle|.
	\end{equation*}
	Thus, %reusing polyvalent constants $C$ and $C'$,
	\begin{multline*}
	\frac{d||\nabla^m v||_{L^2}^2}{dt}+C||\nabla^{m+1} v||_{L^2}^2\leq ||B(u,v)+B(v,u)||_{V^{m-1}}||\nabla^{m+1}v||_{L^2}
	\\
	+||f||_{V^{m-1}}||\nabla^{m+1} v||_{L^2}.
	\end{multline*}
	Therefore, by virtue of Lemma $\ref{estbhaut}$, we get
	\begin{equation*}
	\frac{d||\nabla^m v||_{L^2}^2}{dt}+C||\nabla^{m+1} v||_{L^2}^2\leq C'||u||_{V^{m}}||\nabla^{m+1} v||_{L^2}^2+||f||_{V^{m-1}}||\nabla^{m+1} v||_{L^2},
	\end{equation*}
	and applying Young's inequality to the last right-hand term,
	\begin{equation*}
	\frac{d||\nabla^m v||_{L^2}^2}{dt}+C||\nabla^{m+1} v||_{L^2}^2\leq C'||u||_{V^{m}}||\nabla^{m+1} v||_{L^2}^2+C'||f||_{V^{m-1}}^2.
	\end{equation*}
	Now, Gronwall's lemma allows us to conclude that
	for  $u\in U_{m-1}$, the solution $v\in C([0,T],V^m)\cap L^2([0,T],V^{m+1})$ exists and that  the norm of the inversion 
	%its norm 
	depends on $||u||_{L^2([0,T],V^{m})}$, thus on $||u||_{U_{m-1}}$ alone.
	Finally Corollary $\ref{cestb}$ allows us to conclude that $v\in U_m(T)$.
\end{proof}

We now may prove the existence and analyticity of operator $Sol$:

\begin{proof}[Proof of Theorem \ref{AnalSol}]
	The assertion follows from the real-analytic part of Theorem $\ref{invlocalunif}$ where $f=\mathcal{L}$ and $E=U_m(T')$, $F=V_m\times E_{m-1}(T')$, $Z=B_{V^m(M)}\times\mathcal{E}_{m-1}(T')$. Indeed, by Lemma $\ref{bilin}$ the mapping $\mathcal{L}$ is analytic and by Lemma $\ref{injecL}$ it is injective. By the last assertion of Theorem $\ref{Petcu}$, for any $(u_0,f)\in Z$ a solution $\mathcal{L}^{-1}(u_0,f)=:v$ exists, and $||u||_{U_m}\leq R(T',C^*,m)$. Let us choose $U=B_{U_m}(R+1)$. 
	
	Then \textbf{a)} in Theorem $\ref{invlocalunif}$ holds with $r=1$. Estimate \textbf{b)} follows from Corollary $\ref{cestb}$ (which is valid both for complex and real analytic functions $u(t,x)$), and estimate \textbf{c)} follows from the same Corollary. Finally \textbf{d)} is the assertion of Lemmas $\ref{invL}$, $\ref{invLbis}$. Now application of Theorem $\ref{invlocalunif}$ implies the existence and analyticity of the mapping $\eqref{eqSol}$. Its Lipschitz constant is bounded by
	\begin{equation*}
		\sup_{||u_0||_{V^m}\leq M,f\in\hat{E}_{m-1}(T',\rho)}||dSol(u_0,f)||_{V^m\times E_{m-1}\mapsto U_m}\leq L(M,T,m).
	\end{equation*}	
\end{proof}

\section{Study of the generated discrete-time process}
%\begin{remark}
%     The realisations of $f$ belonging to $L^2([0,T],H^1)$, the latter generates a stochastic process $\mathscr{F}_k$-fitted, with value in $C([0,T],V)$
%\end{remark}
%Now we shall consider the stochastic process thus generated.
Now let us take for $f$ in $\eqref{init}$ a function, defined piecewise in time:
\begin{equation*}
f:=\sum_{k=0}^\infty \mathbbm{1}_{[k,k+1[}(t)\eta_k(t-k),
\end{equation*}
and such that
\begin{equation*}
f\in \mathcal{E}_{m}(\infty), \text{ for some }  m\geq 2%\quad \forall T>0.
\end{equation*}
By Theorem $\ref{Petcu}$ and Corollary $\ref{corespace}$ the primitive equations $\eqref{init}$,%$\eqref{cb2}$, 
$\eqref{ci}$ have a unique a solution $v\in U_m(T)$ for all $T>0$ %$v\in C(\mathbb{R}^+,V^m)\cap L^2_{loc}(\mathbb{R}^+,V^{m+1})$  
 %with right-hand term $f$ and 
 for any $v_0\in V^m$.
%using $f$ with realizations in $L^{\infty}(\mathbb{R}^+,V^m)$, we create $v\in C(\mathbb{R}^+,V^m)\cap L^2_{loc}(\mathbb{R}^+,V^{m+1})$ the solution of the primitive equations with right-hand term $f$ and some $v_0\in V^m$.
%We sample the values of said process on integer thus creating a process $(v_k)_{k\in\mathbb{N}\cup\{0\}}$, where
We calculate the values of the solution $v(t)$ at integer moments of time,
%\begin{equation*}
%v_k(t):=v(t+k),\quad 0\leq t\leq 1.
%\end{equation*}
\begin{equation*}
	v_k=v(k),\quad v_k\in V^m,
\end{equation*}
%
%
%The $\eta_k$ being independant identically distributed variables our process is a Markov chain
%: indeed
%\begin{equation*}
%\mathcal{D}(v_k|v_{k-1},\dots,v_0)= \mathcal{D}(v_k|v_{k-1})
%\end{equation*}
%(where $\mathcal{D}$ is the law of a random value)
%as $v_k(t)$ is the value at time $t$ of the solution of our primitive equations with initial value $v_{k-1}$ and right-hand term $\eta_k$.
%
and define its transition operator as
\begin{equation}
\label{Systeme}
S:(v_n,\eta_{n})\mapsto v_{n+1}, \quad S:V_m\times \mathcal{E}_{m-1}(1)\mapsto V_m.
%\tag{System}
\end{equation}
%we name $\mu_n$ its law: $\mu_n:=\mathcal{D}(v_n)$ (then $\mu_n$ is a law on the space of trajectories from $[0,1]$ to $V^m$).
%We shall now prove that said process is mixing
Let us now study the discrete time process obtained by iterating the applications of this operator.
\subsection{Invariant subset for $S$ and its analyticity}
% Using the notations of theorems $\ref{Absorbe}$ and $\ref{AnalSol}$, we define for $T=T(R)$, $M=R$ and $m\geq 2$ and for $\epsilon(M,T,m)>0$ the following subset of $V^m$
%
%%%Ancienne version plus récente de ci-dessous
%Using the notations of Theorems $\ref{Absorbe}$ and $\ref{AnalSol}$, choose $M=R$ and $T=T(R)$. For $m\geq 2$ and $\epsilon(M,T,m)>0$ let us construct the following subset of $V^m$:
%\begin{equation*}
%\mathcal{O}_m:=\{v,\exists 0\leq t\leq 2T, v_0\in B_{V^m}(R), \eta\in\hat{E}_{m-1}(2T), v=Sol(v_0,\eta)(t)\}.
%\end{equation*} 
%
%First let us notice that by virtue of Theorem $\ref{Absorbe}$, 
%\begin{equation}
%	\label{inclusions}
%	B_{V^m}(R)\subset \mathcal{O}_m\subset B_{V^m}(K).
%\end{equation}
%
%%%Nouvelle version
Let us fix $m\geq 2$. Using the notation of Theorem $\ref{Absorbe}$ we choose $M=2R\geq2$ and find the corresponding $T(M)\geq 1$ and $K=K(M)\geq 2R$.
 Let us apply the last assertion of Theorem $\ref{AnalSol}$. %with $T:= 2T(M)$ 
We may set some $\epsilon(M,2T)$ such that for $||u_0||_{V^m}\leq 2R$, and for any $f'\in \hat{E}_{m-1}(2T,\epsilon(M,2T)) $, so that %as there exists $f\in\mathcal{E}_{m-1}(2T)$ such that 
\begin{equation*}
	||f-f'||_{E_{m-1}(T')}\leq \epsilon(M,2T), \quad f\in\mathcal{E}_{m-1}(2T),
\end{equation*}  
we have that 
\begin{equation*}
	||Sol(u_0,f')(t)-Sol(u_0,f)(t)||_{V^m}\leq \epsilon(M,2T)L,\quad  1\leq t\leq 2T.
\end{equation*}
Let us fix $\epsilon(M,2T)$ so small $\epsilon L\leq 1$. Then by virtue of Theorem $\ref{Absorbe}$, for any $f'\in\hat{E}_{m-1}(2T,\epsilon(M,2T)) $ we have
\begin{equation}
\label{prox}
\begin{cases}
||Sol(u_0, f')(t)||_{V^m}\leq K+1,& 1\leq t\leq T,
\\
||Sol(u_0, f')(t)||_{V^m}\leq R+1,&  T\leq t\leq 2T.
\end{cases}
\end{equation}
From now on we set $\epsilon=\epsilon(R,M,2T)$, $\hat{E}_{m-1}(2T)=\hat{E}_{m-1}(2T,\epsilon)$ and we define
\begin{equation*}
\mathcal{O}_m:=\{v,\exists 0\leq t\leq 2T, v_0\in B_{V^m}(2R), \eta\in\hat{E}_{m-1}(2T), v=Sol(v_0,\eta)(t)\}.
\end{equation*}
Due to $\eqref{prox}$, 
\begin{equation}
	\label{inclusions}
	B_{V^m}(2R)\subset \mathcal{O}_m\subset B_{V^m}(K+1).
\end{equation}
 Moreover we have the following result:
 \begin{lemma}
 	The set $\mathcal{O}_m$ is open in $V^m$% for the usual topology associated with $V^m$
 \end{lemma}
\begin{proof}
	The set
	\begin{equation*}
	W_m:=\{v\in U_{m}(2T),v =Sol(v_0,\eta), v_0\in B_{V^m}(2R), \eta\in  \hat{E}_{m-1}(2T)\}
	\end{equation*}
	is open in $U_{m}(2T)$ as the inverse image by $\mathcal{L}$ of $B_{V^m}(2R)\times \hat{E}_{m-1}(2T)$ which is an open subset of $V^m\times E_m(2T)$.
	We now study the following mappings:
	\begin{equation}
	\mathcal{I}_\tau:U_{m}(2T)\rightarrow V^m,\quad v\mapsto v(\tau).
	\end{equation}
	Due to Theorem 3.2 of the first chapter in \cite{BOK7} these mappings are continuous surjective.
	Therefore, by virtue of the open mapping theorem applied to the Banach spaces $U_m(2T)$, the sets
	 $\mathcal{I}_{\tau}W_m\subset V^m$ are open for any $\tau\in[0,2T]$ and $\mathcal{O}_m=\cup_{\tau\in[0,2T]} \mathcal{I}_\tau W_m$
	is open as well.
\end{proof}
Moreover
\begin{theorem}
	\label{StabAnaS}
	%There exists $\epsilon>0$  such that 
	The set $\mathcal{O}_m$ is invariant for $S$ when the external force belongs to $\mathcal{E}_{m-1}(1)$. In other words,
	\begin{equation*}
	S(\mathcal{O}_m\times\mathcal{E}_{m-1}(1))\subset\mathcal{O}_{m}.
	\end{equation*}
	Moreover, the mapping $S$ is analytical from $\mathcal{O}_m\times\hat{E}_{m-1}(1)$ into $V^m$.
\end{theorem}
\begin{proof}%{of Theorem $\ref{StabAnaS}$}
	The analycity of $S$ follows from that of the mapping $Sol:B_{V^m}(2R)\times\hat{E}_{m-1}(1)\rightarrow U_{m}(1)$ (cf Theorem $\ref{AnalSol}$).
	
%	 and that of 
%	\begin{equation*}
%	U_{m}(1)\rightarrow V^m,\quad v(\cdot)\mapsto v(1).
%	\end{equation*}
	
	Now let us deal with the stability of $\mathcal{O}_m$
	We consider any $\mathcal{o}\in\mathcal{O}_m$ . Then there exist $ t_0\geq 0$, $\mathcal{o}_0\in B_{V^m}(R)$ and $\eta_0\in \hat{E}_{m-1}(t_0)$ such that
	\begin{equation*}
	\mathcal{o}=Sol(\mathcal{o}_0,\eta_0)(t_0).
	\end{equation*}
	Let $z_0\in \mathcal{E}_{m-1}(t_0)$ be an approximation of $\eta_0$ such that
	\begin{equation*}
		||z_0-\eta_0||_{E_{m-1}(1)}<\epsilon.
	\end{equation*}
	Let us now consider any  $z_1\in \mathcal{E}_{m-1}(1)$ and
	%	\begin{equation*}
	%		||z_1(t)||_{V^{m-1}}\leq C_1, \quad\forall 0\leq t\leq 1,
	%	\end{equation*}
	%	En appliquant le théorème $\ref{AnalSol}$ pour $T=1$, on obtient que $Sol(\mathcal{o},f_1)(1)$ dépend analytiquement de $\mathcal{o}$ et $f_1$.
	%Quant à l'appartenance de $S(\mathcal{o},f_1)=Sol(\mathcal{o},f_1)(1)$ à $\mathcal{O}$,
	undertake a dichotomy depending on the definition of $\mathcal{o}$. 
	%%%%Ancienne version de ci-dessous
%	if $t_0\geq T$, 
%	%it is an immediate consequence of theorem $\ref{Absorbe}$ ,besides $\mathcal{o}\in B_{V^m}(R)$ that
%	since $\mathcal{o}\in B_{V^m}(2R)$, by virtue of theorem $\ref{Absorbe}$,
%	\begin{equation*}
%	S(\mathcal{o},z_1)\in B_{V^m}(2R)\subset\mathcal{O}_m.
%	\end{equation*}
	If $t_0\geq T$ then by virtue of $\eqref{prox}$, $||u(t)||_{V^m}\leq R+1\leq 2R=M$, thus $S(\mathcal{o},z_1)\in \mathcal{O}_m$ by definition of $\mathcal{O}_m$.
	
	If, now, $0\leq t_0\leq T$, we shall write $S(\mathcal{o},z_1)$ under the exact form required in the definition of $\mathcal{O}_m$; to this end let us define 
	\[
	z_2(t)=
	\begin{cases}
	z_0(t),&0\leq t\leq t_0,
	\\
	z_1(t-t_0),& t_0\leq t\leq t_0+1,
	\\
	z_1(1)& t\geq t_0+1,
	\end{cases}
	\]
	and
	\[
	\eta_2(t)=
	\begin{cases}
	\eta_0(t),&0\leq t\leq t_0,
	\\
	z_1(t-t_0),& t_0\leq t\leq t_0+1,
	\\
	z_1(1)& t\geq t_0+1.
	\end{cases}
	\]
	Then $\eta_2\in\hat{E}_{m-1}(2T)$ is the control applied to the trajectory starting from $\mathcal{o}_0$, going through $\mathcal{o}$,  defined up to $t_0+1$, and extended up to $2T$ by $z_1(1)$.
	
	Clearly
	\begin{equation*}
	||\eta_2-z_2||_{E_{m-1}(2T)}=||\eta_0-z_0||_{E_{m-1}(2T)}\leq \epsilon.
	\end{equation*}
	Now since,
	\begin{equation*}
		S(\mathcal{o},z_1)=Sol(\mathcal{o}_0,\eta_2)(t_0+1), \quad \mathcal{o}_0\in B_{V^m}(2R), \eta_2\in \hat{E}_{m-1}(t_0+1),
	\end{equation*}
	we get $S(\mathcal{o},z_1)\in \mathcal{O}_m$ and the proposition is proved.
\end{proof}
A consequence of this is, that for $m\geq 2$, we may reduce the study of the long-time behaviour of trajectories of the system $\eqref{Systeme}$ to the behaviour of the trajectories on the invariant sets $\mathcal{O}_m$ .
\section{Regularization property of the mapping $S$}
%The above obtained analycity result has as an image space $U_3$ which, if we want a ponctual analysis leads us to $C([0,1],H^2)$ whereas we would want a better analycity for $S$, with values in $H^3$.
By virtue of Theorem $\ref{AnalSol}$, for $m\geq 2$, the operator $S$ is analytical from $V^m\times \hat{E}_{m-1}(1)$ with values in $V^m$. Thanks to the following proposition we may raise the regularity of the image set by increasing the smoothness of $f$ in $\eqref{initlin}$.
\begin{proposition}
	\label{Reg}
	For all $m\geq 2$, operator $S$ is analytical from $B_{V^m}(2R)\times \hat{E}_{m}(1)$ to $V^{m+1}$
\end{proposition}
To prove this proposition, we first state the following result
\begin{lemma}
	\label{fimplicites}
	Let $j\geq 2$ be fixed and $u\in U_j$, $h\in E_j$ et $v_0\in V^{j+1}$ be given.
	Then, the solution $v$ of 
	\begin{equation*}
	d\mathcal{L}(u)v=(v_0,h)
	\end{equation*}
	satisfies $v\in U_{j+1}$ and analytically depends on $u\in U_j$, $v_0\in V^{j+1}$ and $h\in E_j$.
\end{lemma}
\begin{proof}
	By virtue of Lemmas $\ref{invL}$ and $\ref{invLbis}$, $v\in U_{j+1}$ exists and is unique for all $u\in U_j$, $h\in E_j$ and $v_0\in V^{j+1}$. To prove its analyticity as a function of those parameters, we apply the implicit function theorem. % such as written in $????$.
	In our case, we set
	\begin{equation*}
	f(u,v):=d\mathcal{L}(u)v.
	\end{equation*}
	Thanks to Lemma $\ref{bilin}$, $f$ is analytical as a mapping%between the following spaces
	\begin{equation*}
	U_j\times U_{j+1}\mapsto V^{j+1}\times E_j.
	\end{equation*}
	Moreover, for any $(u,v)$, its differential $d_vf(u,v)=d\mathcal{L}(u):U_j\times U_{j+1}\mapsto V^{j+1}\times E_j$ is invertible between those spaces as shown by Lemmas $\ref{invL}$ and $\ref{invLbis}$.
	Thus, by virtue of the implicit mapping theorem, for all $\bar{u}\in U_j$ and $(\bar{v_0},\bar{h})\in V^{j+1}\times E_j$, there exists a 
	\begin{equation*}
	\phi: U_{j}\times V^{j+1}\times E_j\supset A\rightarrow U_{j+1},
	\end{equation*}
	defined and analytic on a neighbourhood $A$ of $(\bar{u},\bar{v_0},\bar{h})$ such that 
	\begin{equation*}
	f(u,\phi(u,v_0,h))=(h,v_0), \quad \forall (u,v_0,h)\in A.
	\end{equation*}
	Therefore $v=\phi(u,v_0,h)$ is an analytical function everywhere on $U_{j}\times V^{j+1}\times E_j$.
\end{proof}

\begin{proof}[Proof of Proposition $\ref{Reg}$]
	%For $m=2$
	To prove this we notice that as $S(0,0)=0$, then
	\begin{equation*}
	S(u_0,\eta)=\int_0^1\frac{d}{d\gamma}S(\gamma u_0,\gamma\eta)d\gamma=\int_0^1D_{u_0}S(\gamma u_0,\gamma\eta)u_0d\gamma+\int_0^1D_{\eta}S(\gamma u_0,\gamma\eta)\eta d\gamma,
	\end{equation*}
	and set $u_\gamma:=S(\gamma u_0,\gamma\eta)$.
	%The mapping $Sol$ being analytical between
	Since
	\begin{equation*}
		Sol:B_{V^m}(2R)\times \hat{E}_{m-1}(1)\rightarrow U_m(1),
	\end{equation*}
	is an analytical mapping, then
%	it is obviously analytical between
%	\begin{equation*}
%	B_{V^m}(2R)\times \hat{E}_{m}(1)\rightarrow U_m(1);
%	\end{equation*}
	  $u_\gamma$ analytically depends on $(\gamma u_0,\gamma \eta)$ for $(u_0,\eta)\in B_{V^m}(2R)\times \hat{E}_{m}(1)$ and $0\leq \gamma\leq 1$  (we use that, $\hat{E}_{m}(1)$ is star-shaped).

Now let us consider  $\int_0^1D_{\eta}S(\gamma u_0,\gamma\eta)\eta d\gamma$ and notice that $D_{\eta}S(\gamma u_0,\gamma \eta)\eta$ is the value at time $t=1$ of a solution $v_\gamma$ of
\begin{equation*}
\frac{dv_\gamma}{dt}-\Delta v_\gamma+ B(u_\gamma,v_\gamma)+B(v_\gamma,u_\gamma)=\eta,\quad v_\gamma(0)=0.
\end{equation*}
In other words, it satisfies the equation $d\mathcal{L}(u_\gamma) v_\gamma=(\eta ,0)$.
As $\eta\in E_m(1)$ and $u_\gamma\in U_m(1)$ we get, applying Lemma $\ref{fimplicites}$, that $v_\gamma\in U_{m+1}(1)$ analytically depends on $u_\gamma$ and $\eta$, and therefore on $u_0\in B_{V^m}(2R)$ and $\eta\in\hat{E}_{m}(1)$.

%Finally the mapping which matches $v$ with $v(1)$ is obviously linear from $U_{m+1}(1)$ into $V^{m+1}$ 

Thus $D_{\eta}S(\gamma u_0,\gamma\eta)\eta$ is analytic with values in $V^{m+1}$ and by integration the mapping $\int_0^1D_{\eta}S(\gamma u_0,\gamma\eta)\eta d\gamma$ is analytic from $B_{V^m}(2R)\times \hat{E}_{m}(1)$ to $V^{m+1}$. 

To study the integral $\int_0^1D_{u_0}S(\gamma u_0,\gamma\eta)u_0d\gamma$ we note that $D_{u_0}S(\gamma u_0,\gamma \eta)u_0$ is the value at $t=1$ of the solution $\tilde{v}_\gamma$ of
\begin{equation*}
\frac{d\tilde{v}_\gamma}{dt}-\Delta \tilde{v}_\gamma+ B(u_\gamma,\tilde{v}_\gamma)+B(\tilde{v}_\gamma,u_\gamma)=0,\quad \tilde{v}_\gamma(0)=u_0.
\end{equation*}
In other words, $\tilde{v}_\gamma$ verifies $d\mathcal{L}(u_\gamma) \tilde{v}_\gamma=(0,u_0)$.
As $u_0\in V^m$, applying Lemma $\ref{fimplicites}$ with $j=m-1$, we get that $\tilde{v}_\gamma\in U_m(1)$ is an analytical function of $u_0\in V^m$ and $u_\gamma\in U_m(1)$. Now considering $\hat{v}_\gamma=t\tilde{v}_\gamma$ we see $\hat{v}_\gamma$ satisfies %is the solution 
\begin{equation*}
\frac{d\hat{v}_\gamma}{dt}-\Delta \hat{v}_\gamma+ B(u_\gamma,\hat{v}_\gamma)+B(\hat{v}_\gamma,u_\gamma)=\tilde{v} d\gamma,\quad \hat{v}_\gamma(0)=0.
\end{equation*}
Therefore,
\begin{equation*}
d\mathcal{L}(u_\gamma, \hat{v}_\gamma)=(0,\tilde{v}_\gamma).
\end{equation*}
Now, since $(0,\tilde{v}_\gamma)\in V^{m+1}\times E_m(1)$, then by virtue of Lemma $\ref{fimplicites}$ we have that $\hat{v}_\gamma\in U_{m+1}$ is an analytical function of $(u_\gamma,\tilde{v}_\gamma)$. Besides, as $\hat{v}_\gamma(1)=\tilde{v}_\gamma(1)$, then $\int_0^1D_{u_0}S(\gamma u_0,\gamma\eta)u_0d\gamma$ is analytical from $B_{V^m}(2R)\times \hat{E}_{m}(1)$ to $V^{m+1}$.
\end{proof}
\subsection{Dissipativity of the system}
%En reprenant $(H2')$ et ses notations
% First we notice that, applying to $\eqref{init}$ the $L_2$-scalar product with $v$, 
% %and thanks to lemma $\ref{intpart}$ and to formula $\ref{ortho}$ 
% we get the following estimate
%\begin{equation*}
%\frac{d||v||_{L^2}^2}{dt}+||\nabla v||_{L^2}^2=0,\quad  v=Sol(u,0).
%\end{equation*}
%Thus by virtue of Poincaré's inequality and of Gronwall's lemma we get that there is a $\kappa>0$ such that
%\begin{equation*}
%||S(u,0)||_{L^2}\leq e^{-\kappa}||u||_{L^2}.
%\end{equation*}
First let us prove the following result:
\begin{lemma}
	\label{augmnorm}
	Let $m\geq 2$, %%% PAS VRAI POUR m\geq 1 car alors Sol(u_0,0) pas dans U_2 
%	as
%	\begin{equation*}
%		||S(u_0,0)||_{L^2}\leq e^{-\kappa} ||u_0||_{L^2},\quad \forall u_0\in\mathcal{O}_m,%\in B_{V^m}(2R)
%	\end{equation*}
%	is valid, 
	then there exists a $\mathcal{C}_m$ such that for all $u_0\in\mathcal{O}_m$%\in B_{V^m}(2R)$ 
	\begin{equation}
	\label{contrac}
	||S(u_0,0)||_{V^m}\leq \mathcal{C}_m||u_0||_{L^2}.
	\end{equation}
\end{lemma}
\begin{proof}
	%We shall only establish this result for $m=2$, the other cases being proven in a similar fashion.
	We have
	\begin{equation*}
	S(u_0,0)=\int_0^1 D_uS(\gamma u_0,0)u_0d\gamma,
	\end{equation*}
	because $S(0,0)=0$.
	Therefore it suffices to show that $||D_{u}S(\gamma u_0,0)u_0||_{V^m}\leq \mathcal{C}||u_0||_{L^2}$, for $0\leq\gamma\leq 1$.
	%As $\mathcal{O}_m\subset B_{V^m}(K)$
	Thus as  $u_0\in\mathcal{O}_m$, using $\eqref{inclusions}$ and applying Theorem $\ref{Petcu}$ we get that there is a constant $A(K)$ such that
%	\begin{equation*}
%	||Sol( \gamma u_0,0)||_{V^m}(t)\leq A(K),\quad\forall t\geq 0.
%	\end{equation*}
	\begin{equation*}
	||Sol( \gamma u_0,0)||_{V^m}(t)\leq A(K),\quad\forall t\geq 0.
	\end{equation*}
	Moreover
	$D_{u}S(\gamma u_0,0)u_0$ is the value at time $t=1$ of a solution $\tilde{v}$ of
	\begin{equation*}
	d\mathcal{L}(Sol(\gamma u_0,0))\tilde{v}=(u_0,0).
	\end{equation*}
	Thus, by virtue of Lemma $\ref{invL}$ with $r=0$,
	\begin{equation*}
	||\tilde{v}(t)||_{U_0}\leq C(K) ||u_0||_{L^2}.
	\end{equation*}
	%	\begin{equation*}
	%		\int_0^t||\nabla\tilde{v}||_{L^2}\leq C'(R)||u_0|| \quad\forall t
	%	\end{equation*}
	Then, reusing the notation $\hat{v}(t)=t\tilde{v}(t)$ we see that $\hat{v}(t)$ is a solution of
	\begin{equation}
	\label{linchapv}
	d\mathcal{L}(Sol(\gamma u_0,0))\hat{v}=(0,\tilde{v}).
	\end{equation}
	Whence, using Lemma $\ref{invL}$ for $r=1$ we get
	\begin{equation*}
	||\hat{v}(t)||_{U_1}\leq C'(K) ||u_0||_{L^2}.
	\end{equation*}
	More particularly, for a given $t_1<1$
	\begin{equation*}
	||\tilde{v}(t_1)||_{V^1}\leq C'(K) ||u_0||_{L^2}.
	\end{equation*}
	We shall repeat this argument between $t_1$ and $1$. Then $\tilde{v}(t)$ for $t\in[t_1,1] $ is a solution of an equation similar to
	\begin{equation*}
	d\mathcal{L}(Sol(\gamma u_0,0))\tilde{v}=(0,\tilde{v}(t_1)).
	\end{equation*}
	Arguing as above step we have that
	\begin{equation*}
		||D_{u}S(\gamma u_0,0)u_0||_{V^2}\leq \mathcal{C}(K)||u_0||_{L^2}.
	\end{equation*}
	Therefore,
	\begin{equation*}
		||S(u_0,0)||_{V^2}\leq \mathcal{C}_2(K)||u_0||_{L^2},
	\end{equation*}
	and more generally 
	\begin{equation*}
		||Sol(u_0,0)(t)||_{V^2}\leq C_2(K,t)||u_0||_{L^2},
	\end{equation*}
	for any $0<t<1$.
	Now, for any $m>2$ iterating the process on $m$ time intervals $([t_i,t_i+1])_{0\leq i\leq m+1}$ such that $t_0=0$ and $t_{m+1}=1$, we get the assertion of the lemma. 
\end{proof}

From this, we obtain the following

\begin{proposition}
	\label{dissip}
	For $m\geq 2$ consider a norm %the norms
	\begin{equation*}
	||\cdot||'_{V^m}:=||\cdot||_{L^2}+\delta||\cdot||_{V^m},% \quad m\in\mathbb{N}\cup\{0\},
	\end{equation*}
	 equivalent to $||\cdot||_{V^m}$. Then, there exist $\delta=\delta_m >0$, and $0<\gamma<1$  such that the system $\eqref{Systeme}$ abides by the following dissipativity rule:
	\begin{equation*}
	||S(u_0,0)||'_{V^m}\leq \gamma||u_0||_{L^2},\quad \gamma<1,\quad \forall u_0\in\mathcal{O}_m.
	\end{equation*}
	Thus,
	\begin{equation*}
		||S(u_0,0)||'_{V^m}\leq \gamma||u_0||'_{V^m},\quad \gamma<1,\quad \forall u_0\in\mathcal{O}_m.
	\end{equation*}
\end{proposition}
\begin{proof}
	First let us state a simple contraction result for our system: applying to $\eqref{eqinit}$ the scalar product with $v$ and thanks to Lemma $\ref{intpart}$ and to Formula $\ref{ortho}$ we get the following estimate
	\begin{equation*}
	\frac{d||v||_{L^2}^2}{dt}+||\nabla v||_{L^2}^2\leq 0,\quad \text{ for } v=Sol(u_0,0).
	\end{equation*}
	Thus by virtue of Poincaré's inequality and following Gronwall's lemma we get that there is a $\kappa>0$ such that
	\begin{equation*}
	||S(u_0,0)||_{L^2}\leq e^{-\kappa}||u_0||_{L^2}.
	\end{equation*}
	Then by virtue of lemma $\ref{augmnorm}$ we get that 
	\begin{equation*}
	||S(u_0,0)||'_{V^m}= e^{-\kappa}||u_0||_{L^2}+\delta C_m ||u_0||_{L^2}= (e^{-\kappa}+\delta C_m)||u_0||_{L^2}\leq \gamma ||u_0||_{L^2},\quad \gamma<1
	\end{equation*}
	for $\delta= \frac{\gamma-e^{-\kappa}}{C_m}$ where $\gamma\in (e^{-\kappa},1)$.
\end{proof}
%%% RAISONNEMENT NORME EQUIVALENTE

%Bearing this in mind we set the following norm
%\begin{equation*}
%||\cdot||'_{V^2}:=||\cdot||+\epsilon||\cdot||_{V^2}
%\end{equation*}
%Using it we state that $S(\cdot,0)$ is a contraction.
%Indeed,
%\begin{equation*}
%||S(u_0,0)||'_{V^2}= ||S(u_0,0)||+\epsilon ||S(u_0,0)||_{V^2}
%\end{equation*}
%so thanks to our hypothesis and to the above lemma, we get
%\begin{equation*}
%||S(u_0,0)||'_{V^2}= e^{-\kappa}||u_0||+\epsilon C_R ||u_0||= (e^{-\kappa}+\epsilon C_R)||u_0||\leq \gamma ||u_0||,\quad \gamma<1
%\end{equation*}
%for a well-chosen $\epsilon$; finally, let's notice that
%\begin{multline*}
%||S(u_0,\eta)||'_{V^2}\leq ||S(u_0,0)||'_{V^2}+\int_0^1 ||D_{\eta}S(u_0,t\eta)\eta||_{V^2} dt
%\\
%\leq  ||S(u_0,0)||'_{V^2} ||S(u_0,\cdot)||_{C^1(E^1,V^2)}||\eta||_{V^1} \quad\forall (u_0,\eta)\in \mathcal{O}_2\times \hat{E}_1(1)
%\end{multline*}
%Hence the result

\subsection{Non degeneracy of operator $S$}
The aim of this section is to prove %Proposition $\ref{nondegen}$ which states 
the density of the range of the operator $D_{\eta}S(u_0,\eta)$, as an application between $E_m$ and $V^m$, for any $m\geq 2$ and any $(u_0,\eta)\in \mathcal{O}_m\times \hat{E}_{m-1}$. For simplicity's sake, in what follows we shall only establish the result in the case $m=2$. To this end we will first study operators
\begin{equation*}
S_{t_1}^{t_2}:V^2\rightarrow V^2, \quad v\mapsto v',\quad 0\leq t_1\leq t_2,
\end{equation*}
where $v'$ is the value at $t_2$ of a solution $v(t)$ of equation $d\mathcal{L}^2(u) v=0$, $u=Sol(u_0,\eta)\in U_2(1)$, with initial value $v(t_1)=v$. 

Now, for any $0\leq t_1<t_2\leq 1$, let us determine the dual operator of $S_{t_1}^{t_2}$ in $(V^2,||\cdot||_{V^0})$. Let us start with the following result
\begin{lemma}
	For all $u,v,w\in U_2(1)$, we have the following equality
	\begin{equation*}
	\langle B(u,v)+B(v,u),w\rangle =\langle v, \mathbb{B}_u(w)\rangle,
	\end{equation*}
	where
	\begin{equation*}
	\mathbb{B}_u(w):=\mathfrak{P}\left(-B(u,w)+(d_2u^*)w+ \int_{-h}^z\nabla_2\left(w\frac{\partial u}{\partial z}\right)(x,y,\xi)d\xi\right),
	\end{equation*}
	and
	\begin{equation*}
	((d_2u^*)w)^i:=\sum_{l=1}^2 w_l\partial_i u_l \text{ for } i=1,2.
	\end{equation*}
\end{lemma}
\begin{proof}
	First let us notice that thanks to Lemma $\ref{intpart}$, we have $\langle B(u,v+w),v+w\rangle=0$, so that
	\begin{equation*}
	\langle B(u,v),w\rangle= -\langle v,B(u,w)\rangle,\quad \forall u,v,w\in V^2.
	\end{equation*}
	Now let us examine $\langle B(v,u), w\rangle$. Concerning its first term, we have that
	\begin{equation*}
	\langle (v\cdot\nabla_2)u,w\rangle=\int_{\mathbb{O}}\sum_{i=1}^2v_l\Big(\sum_{l=1}^{2}w_l\partial_iu_l\Big)=\langle v,(d_2u^*)w\rangle.
	\end{equation*}
	Then we study the second term of $\langle B(v,u), w\rangle$. We have
	\begin{equation*}
	\langle \left(\int_{-h}^z \diver_2 v d\xi\right)\frac{\partial u}{\partial z},w\rangle=
	\int_{\mathbb{T}_L^2}dxdy\left[\int_{-h}^0\left(\int_{-h}^z\diver_2 vd\xi\right)\frac{\partial u}{\partial z}\cdot wdz\right].
	\end{equation*}
	Writing $\frac{\partial u}{\partial z}\cdot w=\frac{\partial}{\partial z}\left(\int_{-h}^z\frac{\partial u}{\partial z}\cdot wd\xi\right)$ we get
	\begin{equation*}
	\langle \left(\int_{-h}^z \diver_2 v d\xi\right)\frac{\partial u}{\partial z},w\rangle=-\int_{\mathbb{T}_L^2}dxdy\left[\int_{-h}^0\diver_2 v\left(\int_{-h}^z\frac{\partial u}{\partial z}\cdot wd\xi\right)dz\right],
	\end{equation*}
	because $\int_{-h}^0\diver_2 vd\xi=0$. %%% IPP en intégrale fonctionne !
	Thus
	\begin{equation*}
	\langle \left(\int_{-h}^z \diver_2 v d\xi\right)\frac{\partial u}{\partial z},w\rangle=\int_{\mathbb{O}}v(\zeta)\nabla_2\left(\int_{-h}^z\frac{\partial u}{\partial z}\cdot wd\xi\right)(\zeta)d\zeta. 
	\end{equation*}
\end{proof}

Now let us study the equation, generated by the dual operator $\mathbb{B}_u$.

\begin{lemma}
	\label{lemmadualin}
	If $u\in U_2(1)$,  and $0\leq t_1\leq t_2\leq 1$ then for $r=0,1,2$ and $w''\in V^r$, the problem
	\begin{equation}
	\label{dualin}
	\frac{dw}{dt}+\Delta w-\mathbb{B}_{u(t)}(w)=0,\quad t_1\leq t\leq t_2,\quad w(t_2)=w'',
	\end{equation}
 	has a unique solution $w\in U_r([t_1,t_2])$.
\end{lemma}

\begin{proof}
	In order for $w$ to satisfy $\eqref{dualin}$ it is necessary and sufficient that $\Bar{w}:t\mapsto w(t_2+t_1-t)$ verify equation
	\begin{equation*}
	\frac{d\Bar{w}}{dt}-\Delta \bar{w}+\mathbb{B}_{\Bar{u}}(\bar{w})=0
	\end{equation*}
	with $\Bar{w}(t_1)=w''$ (and $\Bar{u}(t)=u(t_2+t_1-t)$).
	The analysis of this equation being similar to the one of $\eqref{initlin}$, we straightforwardly get that $w$ is in an unique way determined by $w''$ and that $w\in U_r([t_1,t_2])$.
\end{proof}

%%% Version cas général?
%\begin{proof}
%	In order for $w$ to satisfy $\eqref{dualin}$ it is both necessary and suufficient that $\Bar{w}:t\mapsto w(t_2+t_1-t)$ verify
%	\begin{equation*}
%	\frac{d\Bar{w}}{dt}-\Delta w+\mathbb{B}_u(w)=0
%	\end{equation*}
%	with $\Bar{w}(t_1)=w''$ (and $\Bar{u}(t)=u(t_2+t_1-t)$).
%	
%	
%	
%	The analysis of this equation being similar to the one of $\eqref{initlin}$ we straightforward get that $w$ is in an unique way determined by $w''$ and that $w\in U_2$.
%\end{proof}

We then set
\begin{equation*}
\Bar{S}_{t_2}^{t_1}: V^2\rightarrow V^2,\quad w''\mapsto w',\quad 0\leq t_1\leq t_2\leq 1,
\end{equation*}
where $w'=w(t_1)$ and $w$ is a solution of $\eqref{dualin}$.
%besides, we notice that our trajectories $v\in U_2$ and $w\in U_2$ verify
%At this point we may remind ourselves that $v\in U_2([t_1,t_2])$ is a solution of $d\mathcal{L}^2(u)v=0$; thus
Now let us recall that $S_{t_1}^{t_2}v=v(t_2)$ where $v(t)$, defined for $t_1\leq t\leq t_2$ solves equations $d\mathcal{L}^2(u)v=0$ and $v(t_1)=v_0$. For this $v(t)$ and the solution $w(t)$ of $\eqref{dualin}$, used to define the operator $\bar{S}_{t_2}^{t_1}$, we have:
\begin{equation*}
\langle v,	\frac{dw}{dt}+\Delta w-\mathbb{B}_u(w)\rangle=0.
\end{equation*}
Integrating by parts, we get
\begin{equation*}
	\langle \Delta v-B(u,v)-B(v,u),w\rangle+\langle v,\frac{dw}{dt}\rangle=0,
\end{equation*}
which in turn becomes
\begin{equation*}
\langle\frac{dv}{dt},w\rangle+\langle v,\frac{dw}{dt}\rangle=0.
\end{equation*}
Whence,
\begin{equation*}
\frac{d}{dt}\langle v,w\rangle=0.
\end{equation*}
We then get that for all $0\leq t_1<t_2\leq 1$, 
\begin{equation}
\label{dual}
%\tag{$\star$}
\langle S_{t_1}^{t_2} v',w''\rangle=\langle v', \Bar{S}_{t_2}^{t_1}w''\rangle,\quad\forall v',w''\in V^2,
\end{equation}
i.e, $\Bar{S}_{t_2}^{t_1}$ is the dual operator to $S_{t_1}^{t_2}$ in $(V^2,||\cdot||_{V^0})$.

\begin{proposition}
	\label{nondegen}
	For $m\geq 2$, the range of $D_{\eta}S(u_0,\eta):E_m\rightarrow V^m$ is dense in $V^m$.
	% In other words, said range is dense.
\end{proposition}

\begin{proof}
	We will prove the result only for the case $m=2$. To do this, it is enough to check that the orthogonal complement %Pas clair que ce dernier terme soit utile
	 to the range of $D_{\eta}S(u_0,\eta):E_2(1)\rightarrow V^2$ in $V^2$ is $\{0\}$. 
%	Indeed as said range is a subspace of Hilbert space $V^2$ we would then have
%	\begin{equation*}
%	\overline{Im(D_{\eta}S(u_0,\eta))}=(Im(D_{\eta}S(u_0,\eta))^\perp)^\perp=\{0\}^\perp=V^2.
%	\end{equation*}
	To this end, we show for any $w\in V^2$ that
	\begin{equation*}
	\langle D_{\eta}S(u_0,\eta)h,w\rangle_{2},\forall \text{ } h\in E_2(1)\quad \implies w=0,%\quad \forall w\in V^2,
	\end{equation*}
	in other words we show that $\ker(\mathcal{S}^*)=\{0\}$,
	where $\mathcal{S}^*:V^2\mapsto E_2(1)$ is the dual operator to $D_{\eta}S(u_0,\eta)$. % i.e a continuous linear operator from $V^2$ to $E_2(1)$.
	Let us denote $Sol(u_0,\eta)(t)=u(t)$, $0\leq t\leq 1$ and consider the corresponding operators $S_{t_1}^{t_2}$ and $\bar{S}_{t_2}^{t_1}$.
	As $(D_{\eta}S(u_0,\eta)h$ is the value at time $t=1$ of the solution $v$ of
	\begin{equation*}
	d\mathcal{L}(Sol(u_0,\eta))v=(0,h),
	\end{equation*}
	then, applying Duhamel's formula, we get that
	\begin{multline*}
	\langle D_{\eta}S(u_0,\eta)h,w\rangle_{2}=\int_0^1\langle S_t^1(h(t)),w\rangle_{2} dt=\int_0^1\langle S_t^1(h(t)),\Delta^2 w\rangle_{V^0} dt
	\\
	=\int_0^1\langle h(t),\Bar{S}_1^t\Delta^2 w\rangle_{V^0} dt= \int_0^1\langle h(t),\Delta^{-2}\Bar{S}_1^t\Delta^2 w\rangle_{2}dt, \quad \forall w\in V^4.
	\end{multline*}
	Thus for $w\in V^4$, we may define $\mathcal{S}^*(w)$  as:
	\begin{equation}
	\mathcal{S}^*(w)(t)= \Delta^{-2}\Bar{S}_1^t\Delta^2 w, \quad \forall w\in V^4.
	\end{equation}
	This done, let us characterize the kernel of $\mathcal{S}^*$ in $V^2$.
	We now consider any $w\in \ker(\mathcal{S}^*)\subset V^2$ and approximate it in $V^2$ by a sequence $(w_n)\in (V^4)^{\mathbb{N}\cup\{0\}}$, such that $||w_n||_{V^2}\leq 2||w||_{V^2}$ for all $n\geq 0$.
	Next, for each $n$ let us build a curve $[0,1]\ni t\rightarrow \xi_t^n$ defined by the relation,
	\begin{equation*}
		\Bar{S}_1^t\Delta^2 w_n=\Delta \xi_t^n,
	\end{equation*}
	then $\xi^n_{\cdot}=\Delta S^*(w_n)$. Since $\Delta^2 w_n \in V^0$, then by virtue of Lemma $\ref{lemmadualin}$ with $r=0$, we have that $\xi_t^n \in U_2(1)$. Let us now prove that the set $(\xi_t^n)_{n\in\mathbb{N}\cup\{0\}}$ is bounded in $U_0(1)$.
	The curve $\eta_t^n:=\Bar{S}_1^t\Delta^2 w_n$ satisfies
	\begin{equation*}
	\frac{d\eta_t^n}{dt}+\Delta \eta_t^n-\mathbb{B}_u(\eta_t^n)=0, \quad 0\leq t\leq 1,\quad \eta_1^n=\Delta^2 w_n\in V^0.
	\end{equation*}
	Thus $\xi_t^n=\Delta^{-1}\eta^n_t$ abides by the relation
	\begin{equation*}
	\frac{\partial \xi_t^n}{\partial t}+\Delta \xi_t^n-\Delta^{-1}\big(\mathbb{B}_u(\Delta\xi_t^n)\big)=0,\quad \xi_1^n=\Delta w_n.
	\end{equation*}
	 To study this equation, we revert the time-flow, that is we apply the transformation $t\rightarrow 1-t$, keeping the name $\xi_t^n$ for the transformed variable.
	Then, applying to the obtained equation the scalar product with $\xi_t^n$ 
	 %and reverting the time flow over $[0,1]$ %(il s'agit d'une équation hyperbolique rétrograde)
	, we get
	\begin{equation*}
	\frac{\partial ||\xi_t^n||_{L^2}^2}{\partial t}+||\nabla \xi_t^n||_{L^2}^2\leq C\Big|\langle\Delta^{-1}\mathbb{B}_u(\Delta\xi_t^n),\xi_t^n\rangle\Big|.
	\end{equation*}
	Since, by virtue of Lemma $\ref{estbhaut}$,
	\begin{multline*}
	|\langle \Delta^{-1}\mathbb{B}_u(\Delta\xi_t^n),\xi_t^n\rangle|=|\langle \mathbb{B}_u(\Delta\xi_t^n),\Delta^{-1}\xi_t^n\rangle|
	\\
	=|\langle B(u,\Delta^{-1}\xi_t^n)+B(\Delta^{-1}\xi_t^n,u),\Delta\xi_t^n\rangle|\leq C ||B(u,\Delta^{-1}\xi_t^n)+B(\Delta^{-1}\xi_t^n,u)||_{V^2}||\xi_t^n||_{L^2}
	\\
	\leq C||u||_{V^3}||\nabla\xi_t^n||_{L^2}||\xi_t^n||_{L^2},
	\end{multline*}
	%%% Ancienne version
%	Thus, by virtue of lemma $\ref{estbis}$,
%	\begin{equation*}
%		|\langle \Delta^{-1}\mathbb{B}_u(\Delta\xi_t^n),\Delta\xi_t^n\rangle|\leq C||u||_{V^3}||\xi_t^n||_{V^1}||\Delta \xi_t^n||_{L^2}
%	\end{equation*}
%	so that
	then
	\begin{equation*}
	\frac{\partial ||\xi_t^n||_{L^2}^2}{\partial t}+||\nabla \xi_t^n||_{L^2}^2\leq C||u||_{V^3}||\nabla \xi_t^n||_{L^2}||\xi_t^n||_{L^2}.
	\end{equation*}
	Since $||u||_{E_3}\leq ||u||_{U_2}$, then from here
	%the existence of $\xi_t^n$ in $C([0,T],H^1)\cap L^2([0,T],H^2)$; moreover	
%	
%%% Ancienne version plus récente	
%	We then get that $(\xi_t^n)_n$ is compact in the weak topology of the space $U_0(1)$. Let $(\tilde{\xi}_t^n)_n$ be a subsequence of $(\xi_t^n)_n$ converging to $(\xi_t)$ in $U_0(1)$; this convergence being alos true in $C([0,1],L^2)$, we have $\tilde{\xi}_t^n(1)\rightharpoonup \xi_t (1)$ in $L^2$ whence $\xi_t (1)=\Delta w$. Moreover,
%	$\mathcal{S}^*$ being a continuous operator between $V^2$ and $E_2(1)$,
%	%réputé superflu
%%	\begin{equation*}
%%	\mathcal{S}^* w_n\rightarrow \mathcal{S}^* w=0, \quad \text{dans } E_2
%%	\end{equation*}
%	%therefore,
%	\begin{equation*}
%	\xi_t^n=\Delta \mathcal{S}^* w_n\rightarrow 0,\quad \text{in } E_0(1).
%	\end{equation*}
%	Thus $\xi_t$ the accumulation point of $\xi_t^n$ in $U_1(1)$ follows
%%	\begin{equation*}
%%	\langle \xi_t, \phi\rangle =0 \forall \phi\in U_1
%%	\end{equation*}
%	%therefore, 
%	$\xi_t=0$. 
%%%% Ancienne version 
%%	\begin{equation*}
%%	\xi_t^n=\Delta^{-1}\Bar{S}_1^t\Delta^2 w_n,
%%	\end{equation*}
%%	thus noticing that $\mathcal{S}^*$ is a closed operator, 
%%	\begin{equation*}
%%	\Delta^{-1} \xi_t^n=\mathcal{S}^* w_n\rightarrow \mathcal{S}^* w=0;
%%	\end{equation*}
%%	so that, as $\xi_t$ is a point of accumulation of $\xi_t^n$ in $U_1\subset E_2$,
%%	$\xi_t=0$.
%	Therefore $w=\Delta^{-1}\xi_1=0$ whence
%	\begin{equation*}
%	\ker(\mathcal{S}^*)=\{0\}.
%	\end{equation*}
%
%
%%% Version en U_0
\begin{equation*}
||\xi_t^n||_{U_0}\leq C(||u||_{U_2},||w_n||_{V^2}),\quad \forall n\geq 0.
\end{equation*}
We then get that $(\xi_{\cdot}^n)_n$ is compact in the weak topology of the space $U_0(1)$. Let $(\tilde{\xi}_{\cdot}^n)_n$ be a subsequence of $(\xi_{\cdot}^n)_n$ weakly converging to  some $(\xi_{\cdot})$ in $U_0(1)$.
Since $\xi_{\cdot}^n=\Delta \mathcal{S}^* w_n$ then,
\begin{equation*}
\xi_{\cdot}\leftharpoonup\xi_{\cdot}^n=\Delta \mathcal{S}^* w_n\rightarrow \Delta \mathcal{S}^* w= 0,\quad \text{in } E_0(1).
\end{equation*}
Thus $\xi_{\cdot}=0$ i.e $\xi_{\cdot}^n\rightharpoonup 0$ in $U_0(1)$. From here 
% $U_m\mapsto C(V^m)$ continue topologie forte comme identité convexe donc continue topologie faible théorème 3.1 premier chapitre de Lions Magenes. 
 we have $\tilde{\xi}_{1}^n\rightharpoonup 0$ in $V^0$. But $\xi_1^n=\Delta w_n\rightarrow \Delta w$ in $V^0$. Thus $\Delta w=0$ and $w=0$, 
  whence
\begin{equation*}
\ker(\mathcal{S}^*)=\{0\}.
\end{equation*}
%
%
%%%% Version C^0\capL^2 pas valabe car pas Hilbert.	
%\begin{equation*}
%||\xi_t^n||_{C([0,1],V^0)}+||\xi_t^n||_{E_1(1)}\leq C(||u||_{U_2},||w_n||_{V^2}),\quad \forall n.
%\end{equation*}
%We then get that $(\xi_{\cdot}^n)_n$ is compact in the weak topology of the space $C([0,1],V^0)\cap E_1(1)$. Let $(\tilde{\xi}_{\cdot}^n)_n$ be a subsequence of $(\xi_{\cdot}^n)_n$ weakly converging to  some $\xi_{\cdot}$ in $C([0,1],V^0)\cap E_1(1)$.
%Since $\xi_{\cdot}^n=\Delta \mathcal{S}^* w_n$ then,
%\begin{equation*}
%0\leftharpoonup\xi_{\cdot}^n=\Delta \mathcal{S}^* w_n\rightarrow \Delta \mathcal{S}^* w= 0,\quad \text{in } E_0(1).
%\end{equation*}
%Thus $\xi_{\cdot}=0$ i.e $\tilde{\xi}_{\cdot}^n\rightharpoonup 0$ in $C([0,1],V^0)\cap E_1(1)$. From here 
%% $U_m\mapsto C(V^m)$ continue topologie forte comme identité convexe donc continue topologie faible théorème 3.1 premier chapitre de Lions Magenes. 
% we have $\tilde{\xi}_{1}^n\rightharpoonup 0$ in $V^0$. But $\xi_1^n=\Delta w_n\rightarrow \Delta w$ in $V^0$. Thus $\Delta w=0$ and $w=0$ 
%  whence
%\begin{equation*}
%\ker(\mathcal{S}^*)=\{0\}.
%\end{equation*}
\end{proof}

Now that the deterministic side is settled let us introduce the associated stochastic problem. Henceforth  we will consider the $\eta_n$ in $\eqref{Systeme}$ and the associated $f$ as random variables whose exact form we shall now specify.
%d'où la validation de l'hypothèse de dissipativité $(H2)$ dans le cas du système considéré de $\mathcal{O}\times K$ dans $\mathcal{O}$.
%\part{Stochastic problem}
%To establish the well-posedness of the problem in our case, we refer to Mrs Petcu's article \cite{ART9} which allows us to directly state some high-order estimates. Indeed by using theorem 3.1 in \cite{ART9} we have:  
%\begin{theorem}
%	\label{Petcu}
%	Let $m\geq 2$, $v_0\in V^m$ and $f\in L^{\infty}(\mathbb{R}^+,V^{m-1})$ then equation $\eqref{init}$ has a unique solution in $C(\mathbb{R}^+,V^m)\cap L^2(\mathbb{R}^+,V^{m+1})$
%\end{theorem}
\section{Red noise: definition and property}\label{rednoisesec}
%Now that the deterministic side is settled let us introduce the stochastic problem at hand :

%\subsection{Red noise}
% On a précèdemment étudié notre système dans le cas d'une force extérieure déterministe. On va maintenant se servir de ces résultats pour l'étude du même système dynamique soumis à une force stochastique dont on prend la définition dans $\cite{ART4}$.
%We now suppose that the external force to which the system is subjected is a random process, named red noise, which we shall now undertake to describe.
%Let us call
%\begin{equation*}
%	e'_i:=\lambda_i^{-1/2}e_i
%\end{equation*}
%it is an orthonormal basis $V^1$; 

We consider the following Haar wavelets of argument $t\in[0,1]$:
%as a basis for time functions:

\begin{equation*}
h_{0}(t):=\mathbbm{1}_{[0,1]}(t),
\end{equation*}
and for $j\geq 0$ and $0\leq k<2^j$
\[
h_{j,k}:=
\begin{cases}
0,& t\leq \frac{k}{2^j},
\\ 1,& \frac{k}{2^j}\leq t< \frac{k+1/2}{2^j},
\\ -1,& \frac{k+1/2}{2^j}< t\leq \frac{k+1}{2^j},
\\ 0,& t\geq \frac{k+1}{2^j}.
\end{cases}
\]
Classically the $\{(2^{j/2}h_{(j,k)}),\quad j\geq 0, 0\leq k<2^j\}$ make up an orthonormal basis of $L^2([0,1],\mathbb{R})$; for the proof one may see \cite{BOK9}. Consequently
\begin{lemma}
	\label{bases}
	%The family $(2^{j/2}h_{(j,k)})$ makes up an orthonormal basis of $L^2([0,1],\mathbb{R})$.
	
	The set $(2^{j/2}h_{(j,k)}\lambda_i^{-m/2}e_i)$, $j\geq 0$, $0\leq k<2^j$, $i\geq 1$ is an orthonormal basis of  $L^2([0,1],V^m)$.
\end{lemma}
%\begin{proof}
%	By virtue of Poincaré's inequality —which stands on $V^m$— we may change the scalar product on this space so that it is defined the following way:
%	\begin{equation*}
%	\langle u,v\rangle_{V^m}:=\langle (\nabla)^m u,(\nabla)^m v\rangle=\langle(-\Delta)^m u,v\rangle
%	\end{equation*}
%	besides, by definition of $e_i$,
%	\begin{equation*}
%	\langle \lambda_i^{-m/2} (-\Delta)^m e_i,\lambda_j^{-m/2} e_j\rangle= \frac{\lambda_i^{m/2}}{\lambda_j^{m/2}}\langle e_i,e_j\rangle= \delta_{i,j}
%	\end{equation*}
%	where $\delta_{i,j}$ is Kronecker's symbol, hence our result.
%\end{proof}

%This done we study the set $\{\eta_k(t), 0\leq t\leq 1\}$
%%=\sum_i \eta_k^i(t)e_i)$, $0\leq t\leq 1$
% of stochastic processes defined the following way: let

Define a random process $\tilde{\eta}(t)\in V$ for $0\leq t\leq 1$ as follows:

\begin{align*}
\label{eta}
&\tilde{\eta}(t):=\sum_{i=0}^{\infty}b_i\tilde{\eta}^i(t)\lambda_i^{-m/2}e_i ,\quad 0\leq t\leq 1,
%\tag{$\eta$}
\\
&\tilde{\eta}^i(t):=\Big(\xi^i_{0}(t)h_{0}(t)+\sum_{j=0}^\infty c^i_j\sum_{k=0}^{2^j-1}\xi^i_{j,k}h_{j,k}(t)\Big).
\notag
\end{align*}
Here the real random variables $\xi^i_{0}$ and $\xi^i_{j,k}$ are independent, identically distributed (i.i.d) such that $|\xi^i_{0}|\leq1$ and  $|\xi^i_{j,k}|\leq 1$ for any $\omega$. In addition, we assume that the coefficients $b_i$ and $c^i_j$ which are all non-zero verify:
\begin{equation*}
%\sum_{i=1}^{\infty} b_i<\infty,\quad \sup_i\sum_{j=0}^{\infty}(c^i_j)^2<\infty\quad  
%\big(\sup_i\sum_j 2^{j/2}|c^i_j|\big)^2\big(\sum_i b_i^2\big)<\infty
\sum_{i=0}^\infty b_i^2\big(1+\big(\sum_{j=0}^\infty c^i_j\big)^2\big)<\infty.
\end{equation*}
Now let us consider a sequence $\eta_k=\sum_{i=0}^{\infty}b_i\eta_k^i(t)\lambda_i^{-m/2}e_i$ of independent copies of the process $\tilde{\eta}$ and set
%%%% Nouvelle ancienne version
%\begin{equation}
%\label{defgene}
%f(t):=\sum_{i=1}^{\infty}f^i(t)\lambda_i^{-m/2}e_i, \quad f^i(t)=\sum_{k=0}^\infty \mathbbm{1}_{[k,k+1[}(t)\eta_k^i(t-k).
%%= \sum_{k=0}^{\infty}\mathbbm{1}_{[k,k+1[}(t)\eta_k(t-k)
%\end{equation}

\begin{equation}
\label{defgene}
f(t):=\sum_{k=0}^{\infty}\mathbbm{1}_{[k,k+1[}(t)\eta_k(t)=\sum_{k=0}^{\infty}\mathbbm{1}_{[k,k+1[}(t)\sum_{i=1}^\infty b_i\tilde{\eta}_k^i\lambda_i^{-m/2}e_i.
%= \sum_{k=0}^{\infty}\mathbbm{1}_{[k,k+1[}(t)\eta_k(t-k)
\end{equation}

%%% Ancienne version
%\begin{equation}
%\label{eta}
%\tilde{\eta}(t):=\sum_{i=0}^{\infty}b_i\Big(\xi^i_{0}(t)h_{0}(t)+\sum_{j=0}^\infty c^i_j\sum_{k=0}^{2^j-1}\xi^i_{j,k}h_{j,k}(t)\Big)\lambda_i^{-m/2}e_i ,\quad 0\leq t\leq 1
%\tag{$\eta$}
%\end{equation}
%where the real random variables $\xi^i_{0}$ and $\xi^i_{j,k}$ are independant, identically distributed such that $|\xi^i_{0}|\leq1$ and  $|\xi^i_{j,k}|\leq 1$ almost surely; where in addition, the coefficients verify:
%\begin{equation*}
%\sum_{i=1}^{\infty} b_i<\infty,\quad \sup_i\sum_{j=0}^{\infty}(c^i_j)^2<\infty,\quad \big(\sup_i\sum_j 2^{j/2}|c^i_j|\big)^2\big(\sum_i b_i^2\big)<\infty
%\end{equation*}
%we now set a sequence $\eta_k$ of independant, identically distributed copies of $\tilde{\eta}$.
%Function $f$ is then defined by
%
%\begin{equation}
%\label{defgene}
%f(t):=\sum_{i=1}^{\infty}f^i(t)\lambda_i^{-m/2}e_i, \quad f^i(t)=\sum_{k=0}^\infty \mathbbm{1}_{[k,k+1[}(t)\eta_k^i(t-k).
%%= \sum_{k=0}^{\infty}\mathbbm{1}_{[k,k+1[}(t)\eta_k(t-k)
%\end{equation}
%\subsection{Properties of the red noise and of the generated stochastic process}

%We notice that $(\eta_k)$ is a set of independant identically distributed random variables with values is the space of mesurable functions from $\mathbb{R}$ to $V^2$. We purvey this family with the following filtration $\mathscr{F}_k$.
%Moreover, the $f^i$ are independant and take their values in the space of functions from $\mathbb{R}$ to $\mathbb{R}$.
%
%Process $f$ verifies $f^{\omega}(t)\in V^m$.
%we name $\mathscr{F}_k$ the filtration fitted to $(\eta_k)$ ;
Now let us prove that $f$ verifies $f^{\omega}(t)\in V^m$ and study its time and space regularities. %of the process
\begin{lemma}
	\label{momL2bruit}
	Random variable $\tilde{\eta}$  is bounded in $L^{\infty}([0,1],V^m)$, for all $\omega$:% more particularly,
	\begin{equation*}
	\exists C^*>0, \quad  ||\tilde{\eta}(t)||_{V^m}^2<C^*,\quad\forall t\in[0,1],\quad \forall \omega.
	\end{equation*}
\end{lemma}
\begin{proof}
Let us define
\begin{equation*}
\Sigma_j^i(t):= \sum_{k=0}^{2^j-1}\xi_{j,k}^ih_{j,k}(t).
\end{equation*}
Since the supports of the functions $h_{j,k}$ for $0\leq k\leq 2^j-1$ are disjointed, and, as $|h_{j,k}|\leq 1$ for all $t\in[0,1]$ and $|\xi_{j,k}^i|\leq 1$ for all $\omega$, we have that
\begin{equation*}
|\Sigma_j^i(t)|\leq 1\quad\forall t,\quad  \forall \omega.
\end{equation*}
Therefore
\begin{multline*}
||\tilde{\eta}(t)||_{V^m}^2=\sum_{i=0}^{\infty}|b_i|^2\Big(|\Sigma_0^i(t)|+\sum_{j=0}^\infty |c^i_j||\Sigma_j^i(t)|\Big)^2
\\
\leq \sum_{i=0}^{\infty}|b_i|^2\Big(|\Sigma_0^i(t)|^2+\Big(\sum_{j=0}^\infty |c^i_j||\Sigma_j^i(t)|\Big)^2\Big)
\leq \sum_{i=0}^{\infty}|b_i|^2\big(1+\big(\sum_{j=0}^\infty c^i_j\big)^2\big)<\infty,
\end{multline*}
for all $t\in[0,1]$ and all $\omega$
\end{proof}
%%% Ancienne version qui marche: attention les conditions sur le bruit c'est le moment au carré pour le théorème et cette propriété pour Petcu
%\begin{proof}
%	By definition
%	\begin{equation*}
%	||\tilde{\eta}(t)||_{V^m}^2=\sum_{i=0}^{\infty}|b_i|^2\Big(|\xi^i_0(t)||h_{0}(t)|+\sum_{j=0}^\infty |c^i_j|\sum_{k=0}^{2^j-1}|\xi^i_{j,k}||h_{j,k}(t)|\Big)^2
%	\end{equation*}
%	thus, applying Young's inequality
%	\begin{equation*}
%	||\tilde{\eta}(t)||_{V^m}^2\leq 2\sum_{i=0}^{\infty}|b_i|^2\Big(|\xi^i_0(t)|^2|h_{0}(t)|^2+\sum_{j=0}^\infty |c^i_j|^2\sum_{k=0}^{2^j-1}|\xi^i_{j,k}|^2|h_{j,k}(t)|^2\Big).
%	\end{equation*}
%	Thence 
%	%thus neglecting the set of realisation such that  $|\xi^i_{j,k}|\geq 1$, we get
%	\begin{equation*}
%	||\tilde{\eta}(t)||_{V^m}^2\leq \sum_{i=0}^{N}|b_i|^2|h_{0}(t)|^2+\sum_{j=0}^\infty|c_j|^2\sum_{k=0}^{2^j-1}|h_{j,k}(t)|^2,%,\quad a.s
%	\end{equation*}
%	so that
%	\begin{equation*}
%	\sup_{t\in[0,1]} ||\tilde{\eta}(t)||_{V^2}^2\leq \sum_{i=0}^{N}|b_i|^2+\sum_{i=0}^{N}|b_i|^2\sum_{j=0}^\infty|c_j|^2.%, \quad p.s
%	\end{equation*}
%	
%	%%% Version avant si problème
%	\begin{equation*}
%	||\eta_r(t)||_{V^m}^2\leq \sum_{i=0}^{N}|b_i|^2|h_{0}(t)|^2+\sum_{j=0}^\infty|c_j|^2\sum_{k=0}^{2^j-1}|h_{j,k}(t)|^2,%,\quad a.s
%	\end{equation*}
%	so that
%	\begin{equation*}
%	\sup_{t\in[0,1]} ||\eta_r(t)||_{V^2}^2\leq \sum_{i=0}^{N}|b_i|^2+\sum_{i=0}^{N}|b_i|^2\sum_{j=0}^\infty|c_j|^2.%, \quad p.s
%	\end{equation*}
%	%Whence the result.
%\end{proof}
Now we shall study the properties of the stochastic system $\eqref{Systeme}$ where $f$ is given by $\eqref{defgene}$. Namely in the following section we will strengthen the notion of mixing which we first exposed in $\eqref{melsto}$, replacing it with the exponential mixing, and we present a general result on mixing in stochastic systems which we shall apply to $\eqref{Systeme}$ in the subsequent sections.

\section{An abtsract theorem on mixing}
To show that a stochastic process is exponentially mixing, we rely on Theorem 1.3 in \cite{ART7} which we discuss in the following context.
Let $H$ and $E$ be separable Hilbert spaces and $V$ a Banach space with a compact injection into $H$. Let  $(\eta_k)_{k\geq 0}$ be a sequence of independent identically distributed random variables with values in $E$. Let $\bar{S}:H\times E\rightarrow H$ be a continuous operator. Let us fix $u_0\in H$ and consider the random dynamic system in $H$: 
\begin{equation}
\label{Sys}
%\tag{Sys}
u_{k+1}=\bar{S}(u_k,\eta_k),\quad k\geq 0.
\end{equation}  
It defines a Markov chain in the space $H$.
Denote by $(u_k(u_0),k\geq 0)$ a trajectory of this system, equal to $u_0$ when $k=0$.
To proceed we lay out some additional notation. Namely, for a complete separable metric space, $U$ we denote by $C_b(U)$ the space of continuous bounded functions $g$ on $U$ equipped with the sup-norm $||g||_{L^\infty}$. For a function $g\in C_b(U)$, denote by $Lip(g)\leq \infty$ its Lipschitz constant and set: 
\begin{equation*}
	||g||_L:=||g||_{L^\infty}+Lip(g).
\end{equation*}
Moreover, we name $\mathcal{P}(U)$ the space of probability Borel measures on $U$. We equip it with the weak convergence of measures and with the dual Lipschitz distance
\begin{equation*}
		||\mu-\nu||_{L}^*=\sup_{g\in C_b(U), ||g||_L\leq 1} \langle g,\mu-\nu\rangle\leq 2,
\end{equation*}
where $\langle g,\mu\rangle$ stands for the integral of $g$ against $\mu$. It is known that the weak convergence of measures is equivalent to the convergence in the dual-Lipschitz distance (see \cite{BOK3} ). Finally for a random variable $\xi\in U$ we denote by $\mathcal{D}\xi$ its law $\mathcal{D}\xi\in\mathcal{P}(U)$.
We recall that a measure $\mu\in\mathcal{P}(H)$ is stationary for the system $\eqref{Sys}$ if any trajectory $(u_k(u_0))_k$ where $u_0\in H$ is a random variable such that $\mathcal{D}(u_0)=\mu$, satisfies $\mathcal{D}(u_k)=\mu,\forall k$.
In \cite{ART4,ART7}\footnote{In \cite{ART4} a more general result is proved with stronger restrictions on the system}, the following abstract theorem is proved:
%%% Ancienne version Kuksin n'aime pas
%Moreover, we define the dual-Lipschitz distance for measures:
%\begin{definition}
%	Let $U$ with metric $d_U$ be a Polish space, on which we bestow the Borel tribe $\mathcal{U}$; baptise $\mathcal{P}(U)$ the set of all probabilities on $(U,\mathcal{U})$. On those measures we define the dual-Lipshcitz distance:
%	\begin{equation*}
%		||\mu-\nu||_{L}^*=\sup_{g\in C_b(U), ||g||_L\leq 1} \langle g,\mu-\nu\rangle
%	\end{equation*}
%	where $C_b(U)$ is the set of continuous bounded function on $U$ and for any $g\in C_b(U)$,
%	\begin{equation*}
%		||g||_L:=||g||_{L^\infty}+\sup_{0<d_U(x,y)}\frac{|g(x)-g(y)|}{d_U(x,y)}.
%	\end{equation*}
%\end{definition}
%Finally we call $\mathcal{D}\xi$ the law of random variable $\xi$. Then we have  

\begin{theorem}
	\label{KuksinZhang}
	%Let's $S:=S(u_0,\eta)$, our iterator be between generic spaces $H\times E$ into $V$, $V$ being compactly embedded into $H$
	Under the following hypotheses
	\begin{itemize}
		\item $\textbf{(A1)}\quad\textbf{Regularity: }$ $\bar{S}:H\times E\rightarrow V$ is twice continuously differentiable and its derivatives up to the second order are bounded on bounded sets.
		\item $\textbf{(H1)}\quad\textbf{Decomposability and non-degeneracy}$ 
		%of random variables $\eta_k$ 
		there exists an orthonormal basis $(\sigma_j)_{j\geq 1}$ of $E$, such that the random variables $\eta_k$ can be  written in the following way
		\begin{equation*}
		\eta_k:=\sum_{j=1}^\infty b_j\xi_{j,k}\sigma_j,\quad b_j\neq 0,\quad \sum b_j^2< \infty.
		\end{equation*}
		Here the $\xi_{j,k}$ are random i.i.d variables almost surely smaller than $1$ with laws of the form $\rho(r)dr$, where $\rho$ is a Lipschitz function with a non-zero value at zero.
		We name $l:=\mathcal{D}(\eta_k)$ and denote by $\mathcal{K}$ the support of $l$.
		%Moreover the support $\mathcal{K}$ of the law $l$ of $\eta_k$ is compact.
		\item $\textbf{(H2)}\quad\textbf{Dissipativity:}$ for any $\tau>0$, any $u\in H$ and any $\eta\in \mathcal{K}$ we have
		\begin{equation*}
		a) ||\bar{S}(u,\eta)||_{H}\leq \gamma||u||_H+\beta,\quad b) ||\bar{S}(u,0)||_{H}\leq \gamma||u||_H\quad \gamma<1,\beta\geq 0,
		\end{equation*}
		\item $\textbf{(H3)}\quad\textbf{Non-degeneracy: }$ for any $u\in H$ and $\eta_0\in \mathcal{K}$ the application $D_{\eta}\bar{S}(u,\eta_0):E\rightarrow V$ has a dense range in $H$.
	\end{itemize}
	Then the system $\eqref{Sys}$ is exponentially mixing: there is a unique stationary measure $\mu$ which is supported by $B_H(R^*)$ for some $R^*>0$, %and constants $C>0$ and $0<\kappa<1$
	 such that for any $R>0$, for all $u_0\in B_H(R)$, we have
	% naming $(\mu_k)_{k\geq 0}$ the laws for the iterations of our trajectory $(u_k(u_0),k\geq 0)$, we have
	\begin{equation}
	\label{contracmes}
	||\mathcal{D}(u_k(u_0))-\mu||_{L}^*\leq C\kappa^k,
	\end{equation}
	with constants $0<\kappa(R)<1$, $C=C(R)>0$.
\end{theorem}

\section{Reformulation of the hypotheses}

%In order to prove the validity of the hypotheses of theorem \ref{KuksinZhang}, we shall use the following notations
%\begin{itemize}
%	\item $E_j:=E_j(T):=L^2([0,T], V^j)$ which we equip with $||\cdot||_j$
%	\item $U_j:=U_j(T):=\{u\in E_j(T),\frac{\partial u}{\partial t}\in E_{j-2}(T)\}$, we notice that $U_j\hookrightarrow C([0,T],V^{j-1})$
%	\item $\mathcal{E}_j(T):=\{z,||z(t)||_{V^j}\leq C1\}=L^{\infty}([0,T],H^j)$
%	\item for a fixed $\epsilon>0$ we set 
%	\begin{equation*}
%	\hat{E}_{m-1}(\tau):=\mathcal{E}_{m-1}(\tau)+B_{E_{m-1}(\tau)}(0,\epsilon)\subset E_{m-1}(\tau)
%	\end{equation*}
%	and define the following subset of $V^m$:
%	\begin{equation*}
%	\mathcal{O}_m:=\{v=Sol(v_0,\eta)(t), 0\leq t\leq 2T ,v_0\in B_{V^m}(2R), \eta\in  \hat{E}_{m-1}(2T)\}
%	\end{equation*} 
%	\item $\mathcal{L}$ an operator describing equation $\eqref{init}$
%	\begin{equation*}
%	\mathcal{L}:(v) \rightarrow (\mathcal{L}^1v:=\frac{\partial v}{\partial t}-\Delta v+b(v,v),\mathcal{L}^2v:=v(0)),
%	\end{equation*}
%	\item $Sol: (u_0,\eta)\rightarrow u$ solution of $\eqref{init}$, that is the inverse of $\mathcal{L}$,
%	\item $d\mathcal{L}$ linked to the linearized equation by $\eqref{init}$
%	\begin{equation*}
%	d\mathcal{L}(v): w\rightarrow (d\mathcal{L}^1w:=\frac{\partial w}{\partial t}-\Delta w+b(v,w)+b(w,v),d\mathcal{L}^2w:=w(0)),
%	\end{equation*}
%	where $b(v,w):=(v\cdot\nabla_2)w-\int_{-h}^z\diver_2v\frac{\partial w}{\partial z}$.
%	%\item $S:(u_0,\eta)\rightarrow u(1)$ which is coherent with the above given definition of $S$ as a transition operator.
%\end{itemize}

We seek to apply to the system $\eqref{Systeme}$ a variation of Theorem \ref{KuksinZhang}, where the hypotheses are somewhat relaxed. To this end we analyse the demonstration of the theorem: in \cite{ART7} the authors work in two steps. First, they state the existence of an open bounded subset $\mathcal{O}$ of the phase space that is both invariant and absorbant for the process. In their case, $\mathcal{O}=B_H(R)$ for a specific $R$.
Using hypothesis $\textbf{(A1)}$, they define some $\mathfrak{K}(\mathcal{O},\mathcal{K})>0$ such that 
\begin{equation*}
	\bar{S}:\mathcal{O}\times \mathcal{K}\mapsto B_{V}(\mathfrak{K}(\mathcal{O},\mathcal{K})),
\end{equation*}
and name $X$, the completion of $\mathcal{O}\cap B_{V}(\mathfrak{K}(\mathcal{O},\mathcal{K}))$ in $H$.
Then they consider the following restrained system:
%\begin{equation*}
%\bar{S}:\mathcal{O}\times\mathcal{K}\mapsto \mathcal{O},
%\end{equation*}
\begin{equation*}
\bar{S}:X\times\mathcal{K}\mapsto X,
\end{equation*}
%extended into a twice continuously differentiable mapping
%\begin{equation}
%\label{exten}
%\tilde{S}:\mathcal{O}_{\epsilon}\times\mathcal{K}_{\epsilon}\mapsto V
%\end{equation}
%where $\mathcal{O}_{\epsilon}\subset H$, and $\mathcal{K}_{\epsilon}\subset E$ are sets encompassing the $\epsilon$-neighbourhoods of $\mathcal{O}$ and $\mathcal{K}$ respectively, and such that $\tilde{S}$ be bounded in norm $C^2$ over those sets.
%
%%
%\begin{equation}
%\label{exten}
%\tilde{S}:\mathcal{O}\times\mathcal{K}_{\epsilon}\mapsto V
%\end{equation}
%where $\mathcal{K}_{\epsilon}\subset E$ is a set containing the $\epsilon$-neighbourhood of  $\mathcal{K}$  and such that $\tilde{S}$ be bounded in the $C^2$-norm  over $\mathcal{O}\times\mathcal{K}_{\epsilon}$.
where $\bar{S}$ extends to a $C^2$ smooth mapping
%%%Ancienne version de ci-dessous
%\begin{equation}
%\label{exten}
%\tilde{S}:\mathcal{O}\times\mathcal{K}_{\epsilon}\mapsto V.
%\end{equation}
%Here, $\mathcal{K}_{\epsilon}$ is the $\epsilon$-neighbourhood of $\mathcal{K}$ in $E$ and the $C^2$ norm of the mapping $\eqref{exten}$ is bounded.
\begin{equation}
\label{exten}
\tilde{S}:H\times E\mapsto V.
\end{equation}
%enlevé car redondant avec plus bas
%and the $C^2$ norm of the mapping $\eqref{exten}$ is bounded on bounded sets.
In this context we formulate the following reduced hypotheses on $\bar{S}$: %%% différent de $\tilde{S}$, son extension
\begin{itemize}
	\item $\textbf{(A1${}'$)}$: $\tilde{S}$ verifies $\textbf{(A1)}$
	\item $\textbf{(H1)}$
	\item $\textbf{(H2${}'$)}$: $\textbf{(H2) b)}$ holds
	\item $\textbf{(H3${}'$)}$: $\textbf{(H3)}$ holds for $(u,\eta_0)\in\mathcal{O}\times\mathcal{K}$.
\end{itemize}
Then, in their demonstration of Theorem 1.3 (which is the second step of the proof of mixing for the system), they prove the following:
%a slightly more general result which implies $\eqref{contracmes}$:
%
%ils réduisent l'étude du système: en effet la clé de voûte de leur preuve du résultat de mélange (théorème 1.5 dans \cite{ART7}  ) est
\begin{theorem}{ From theorem 1.3 in \cite{ART7}. }
	\label{melnouv}
	Suppose that our sytem verifies $\textbf{(A1${}'$)}$, $\textbf{(H1${}'$)}$, $\textbf{(H2)}$ and $\textbf{(H3${}'$)}$.
	%then there is a measure $\mu$ such that $\eqref{contracmes}$ holds.
	%then, for any $R>0$ and $u,u'\in \mathcal{O}$,  naming$\mu_n:=\mathcal{D}(u_n(u))$ and $\mu'_n:=\mathcal{D}(u_n(u'))$,  we get
	%  		$\mu_n$ et $\mu'_n$ étant les lois des processus stochastiques itérés par $S$ à partir de respectivement  $u$ et $u'$ ,
	 Then the system $\eqref{Systeme}$ has a unique stationary measure $\mu\in\mathbb{P}(\mathbb{O})$ and for any $u,u'\in \mathcal{H}$, we have
	\begin{equation}
	||\mathcal{D}(u_k(u))-\mathcal{D}(u_k(u')||_L^*\leq C||u-u'||_{H}\kappa^k,
	\end{equation}
	for some $0<\kappa<1$, $C>0$ that only depend on $\mathcal{O}$.
\end{theorem}
Moreover we may relax the constraint on $\mathcal{O}$ which instead of being an open ball of $H$ can merely be a bounded open domain of it. 
%Furthermore, the extension $\tilde{S}$ need not be defined, twice differentiable and have a boundary on its $C^2$-norm on the whole $H\times E$. 
Furthermore, there is no need to demand the existence of an extension $\tilde{S}$ as in $\eqref{exten}$. 
It is enough that $\tilde{S}$ be defined, twice differentiable and bounded in $C^2$-norm on $X_\epsilon\times\mathcal{K}_{\epsilon}$, where $X_\epsilon$ and $\mathcal{K}_{\epsilon}$ are the $\epsilon$-neighbourhoods of $X$ and $\mathcal{K}$ respectively in $H$ and $E$.
 
This clarified, we will verify hypotheses $\textbf{(A1${}$')}$, $\textbf{(H1${}$')}$, $\textbf{(H2)}$ and $\textbf{(H3${}$')}$ for our system.
In our case,  $\bar{S}=S$, $H=V^m$, $V=V^{m+1}$, $E=E_m(1)$ and $\mathcal{O}=\mathcal{O}_m$ for some $m\geq 2$. By virtue of Theorem $\ref{Petcu}$ and of Theorem $\ref{StabAnaS}$, this last set is indeed invariant by $S$ and absorbing for it.

\section{Verification of the hypotheses and conclusion}
\label{verifhyp}
We will verify the hypotheses $(A1${}$')$-$(H3)$ with the norm $||\cdot||_{V^m}$ replace by the equivalent norm $||\cdot||'_{V^m}$. Then $(H2${}$')$ holds by Proposition $\ref{dissip}$, and
%We understand that thanks to proposition $\ref{dissip}$, we may validate hypothesis $(H${}$'2)$ for a norm that is equivalent to our initial one. In this equivalent norm 
hypothesis $(A1${}$')$ stands true by virtue of Theorem $\ref{StabAnaS}$, the validity of $(H1)$ is obvious considering the structure of our noise, as for hypothesis $(H3${}$')$ Proposition $\ref{nondegen}$ allows us to verify it here.

Thus we get

%%% Ancienne version
%\begin{theorem}
%	Let us suppose that the noise $f$ be of the form $\eqref{defgene}$, with $m\geq 2$.Then let $(\mu_k(u_0))_{k\in\mathbb{N}\cup\{0\}}$ be the chain of the laws of  $(u_k(u_0))_{k\in\mathbb{N}\cup\{0\}}$ generated on $\mathcal{O}_m$ through $\eqref{Systeme}$, by the primitive $\eqref{init}$ with initial condition $u_0\in \mathcal{O}_m$. This process has a unique stationnary measure $\mu\in\mathcal{P}(V^m)$.
%	Moreover, there exists $C>0$ and $\kappa<1$ such that
%	\begin{equation*}
%	||\mu_k(u_0)-\mu||_{L(V^2)}^*\leq C\kappa^k.
%	\end{equation*}
%	which in turn yields
%	\begin{equation*}
%		\mathbb{E}[g(u_k(u_0))]\xrightarrow[k\rightarrow \infty]{}\int_{V^m} g(u) d\mu(u),\quad \forall g\in C_b(V^m),\quad ||g||_L\leq 1
%	\end{equation*}
%\end{theorem}
\begin{theorem}
	\label{mixing}
	Assume that the force $f$ in $\eqref{init}$ is the red noise $\eqref{defgene}$, and consider the corresponding system $\eqref{Sys}$ in a space $V^m$ with $m\geq 2$.
	Then this system has a unique stationary measure $\mu$. It is supported by the set $\mathcal{O}_m$ and for any random variable $u_0\in\mathcal{O}_m$, we have
	\begin{equation*}
	||\mathcal{D}(u_k(u_0))-\mu||_{L(V^m)}^*\leq C_m\kappa_m^k.
	\end{equation*}
	where $C_m\geq 0$ and $0<\kappa_m<1$.
\end{theorem}

Moreover, by virtue of Theorem $\ref{Absorbe}$, the set $O_m$ is absorbing for our system. Therefore  we have the following corollary to the previous theorem:
%%% Ancienne version plus exacte 
%\begin{corollary}
%	Let us still suppose that the noise $f$ be of the form $\eqref{defgene}$, with $m\geq 2$. Then, for any $R'>0$, and for any random variable $u_0$ with law $\mu_0$ supported in the ball $B_{V^m}(R')$, there exists $k_0(R')\in\mathbb{N}\cup\{0\}$ $C_m(R')>0$ and $\kappa(R')<1$ such that
%	\begin{equation*}
%	||\mu_k(u_0)-\mu||_{L(V^m)}^*\leq C_m(R')\kappa_m^{k-k_0},\quad\forall k\geq k_0.
%	\end{equation*}
%	Since the convergence in the dual-Lipschitz distance is equivalent to the weak convergence of measures, for any random variable $u_0$ as in the assumptions of this corollary, we have
%	\begin{equation*}
%		\mathbb{E}[g(u_k(u_0))]\xrightarrow[k\rightarrow \infty]{}\int_{V^m} g(u) d\mu(u),\quad \forall g\in C_b(V^m),\quad ||g||_L\leq 1,
%	\end{equation*}
%	cf $\eqref{melsto}$.
%\end{corollary}
%\begin{proof}
%	We set $k_0(R)$ as the smallest integer such that $k_0(R)\geq T_0=T(R,C^*)$ with $T$ as defined in Theorem $\ref{Absorbe}$. From this point on the system is in $\mathcal{O}_m$, thus we may apply our above-mixing theorem hence the result.
%\end{proof}

\begin{corollary}
	\label{coromel}
	Let us still suppose that the noise $f$ be of the form $\eqref{defgene}$, with $m\geq 2$. Then, for any $R'>0$, and for any random variable $u_0\in B_{V^m}(R')$, 
	 %with law $\mu_0$ supported in the ball $B_{V^m}(R')$,
	  there exists $C_m(R')>0$ and $\kappa(R')<1$ such that for any $ k\geq 0$,
	\begin{equation}
	\label{ineqmes}
	||\mathcal{D}(u_k(u_0))-\mu||_{L(V^m)}^*\leq C_m(R')\kappa_m^{k}.
	\end{equation}
\end{corollary}
Since the convergence in the dual-Lipschitz distance is equivalent to the weak convergence of measures, for any random variable $u_0$ as in the assumptions of this corollary, we have
\begin{equation*}
\mathbb{E}[g(u_k(u_0))]\xrightarrow[k\rightarrow \infty]{}\int_{V^m} g(u) d\mu(u),\quad \forall g\in C_b(V^m),%\quad ||g||_L\leq 1,
\end{equation*}
cf $\eqref{melsto}$.
\begin{proof}[Proof of Corollary $\ref{coromel}$.]
	We set $k_0(R)$ as the smallest integer such that $k_0(R)\geq T_0=T(R,C^*)$ with $T$ defined in Theorem $\ref{Absorbe}$. When $k\geq k_0(R)$, the system is in $\mathcal{O}_m$, thus we may apply Theorem $\ref{mixing}$ hence the result for $k\geq k_0$ with $C_m(R'):=C_m\kappa^{-k_0}$. Moreover, as $||\mu_k(u_0)-\mu||_{L(V^m)}^*\leq 2$, then increasing $C_m(R')$ if needed we achieve that $\eqref{ineqmes}$ also holds for any $k\leq k_0(R)$. 
\end{proof}

\section{Appendix}

\begin{theorem}[Uniform local inverse theorem]
	\label{invlocalunif}
	Let $E$ and $F$ be two complex Banach spaces and $U$ be an open bounded subset of $E$. Let $r>0$ and
	\begin{equation*}
	f:U\mapsto F,
	\end{equation*}
	is an analytical mapping which is injective.
	Let $Z\subset F$  be such that $Z\subset f(U)$ and \textbf{a)}: $G_r:= f^{-1}(Z)+B_E(r)\subset U$.
	
	Moreover, we assume that :
	\begin{itemize}
		%\item \textbf{a)}: $G_r:= f^{-1}(Z)+B_E(r)\subset U$,
		\item \textbf{b)}: $||df(x)||_{E\mapsto F}\leq K_1, \quad \forall x \in G_r$,
		\item \textbf{c)}: $||d^2f(x)||_{E\times E\mapsto F}\leq K_2, \quad \forall x\in G_r$ and
		\item \textbf{d)}: $\forall x\in f^{-1}(Z)$, $||(df(x))^{-1}||_{F\mapsto E}\leq K_3$.
	\end{itemize}
		
	Then $\exists\rho=\rho(K_1,K_2,K_3,r)$ and $\exists L=L(K_1,K_2,K_3,r)$ such that on the set $Z_{\rho}=Z+B_F(\rho)$ the inverse mapping $f^{-1}:Z_{\rho}\mapsto U$ is well defined and analytic. Moreover $||df^{-1}(z)||_{F\mapsto E}\leq L$, $\forall z\in Z_{\rho}$.
	The assertion remains true if $E$ and $F$ are real Hilbert spaces, the map $f$ is real-analytic and $G_r$ is an $r$-neighbourhood of $f^{-1}(Z)$ in the complexification of the Hilbert space $E$.
\end{theorem}
 %Pour pouvoir faire l'extension tranquillement comme lui même l'a fait remarquer les propriétés concernées s'étendent facilement.
 \begin{proof}
 	Here we shall only prove that $f^{-1}$ may be built globally on $Z_{\rho}$ when its components on the $B_F(z,\rho)$ for all $z\in Z$ are each well defined. To do so, we only need to remark that $f$ being injective, if two balls $B_F(z,\rho)$ have a non-void intersection, then the locally built $f^{-1}$ must coincide on said intersection, by virtue of $f$ being injective. Thus $f$ is unequivocally defined on the reunion $Z_\rho$ of all $B_F(z,\rho)$ with $z\in Z$.
 \end{proof}

\printbibliography
\end{document}